\documentclass[aop,preprint]{imsart}

\usepackage{amsthm,amsmath,amssymb}%,natbib}
\RequirePackage[colorlinks,citecolor=blue,urlcolor=blue]{hyperref}

\arxiv{1405.2639}

\startlocaldefs

\usepackage[margin=1.25in]{geometry}

\usepackage{cleveref}

\theoremstyle{plain}
\newtheorem{thm}{Theorem}%[section]
\newtheorem{lem}[thm]{Lemma}

\theoremstyle{definition}

\theoremstyle{remark}

\newtheorem{nremark}{Remark}

\DeclareMathOperator{\Prtxt}{Pr}

\newcommand{\cO}[1]{\mathcal{O}\left(#1\right)}

\newcommand{\cF}{\mathcal{F}}

\newcommand{\RR}{\mathbb{R}}      % Real numbers
\newcommand{\NN}{\mathbb{N}}      % natural numbers

\newcommand{\evp}[2]{\mathbb{E}_{#2} \left[#1\right]} % expected value operator
\newcommand{\abs}[1]{\left| #1 \right|}

\newcommand{\prp}[2]{\Prtxt_{#2} \left(#1\right)}

\newcommand{\lrp}[1]{\left(#1\right)}

\newcommand{\expp}[1]{\exp \left(#1\right)}

\newcommand*{\qedinpw}{\hfill\ensuremath{\square}} % white box

\endlocaldefs

\begin{document}

\begin{frontmatter}

% "Title of the paper"
\title{Sharp Finite-Time Iterated-Logarithm Martingale Concentration}
%\author{Akshay Balsubramani \\
%UC San Diego 
%\footnote{Email: \texttt{abalsubr@cs.ucsd.edu}}
\runtitle{Nonasymptotic Iterated-Log Concentration}

% indicate corresponding author with \corref{}
 \author{\fnms{Akshay} \snm{Balsubramani}\corref{}\ead[label=e1]{abalsubr@ucsd.edu}\thanksref{t1}}
 \thankstext{t1}{Supported by NSF grant IIS-1162581} 
 \address{Department of Computer Science and Engineering \\
University of California, San Diego \\
La Jolla, California 92093, USA\\ \printead{e1}}
 \affiliation{University of California, San Diego}

\runauthor{Balsubramani}

\begin{abstract}
We give concentration bounds for martingales that are uniform over finite times and extend classical Hoeffding and Bernstein inequalities. 
We also demonstrate our concentration bounds to be optimal with a matching anti-concentration inequality, proved using the same method. 
Together these constitute a finite-time version of the law of the iterated logarithm, and shed light on the relationship between it and the central limit theorem. 
\end{abstract}



\end{frontmatter}

\section{Introduction}
\label{sec:rwbound}

Martingales are indispensable in studying the temporal dynamics of stochastic processes 
arising in a multitude of fields \cite{H72, P09}.
Particularly when such processes have complex long-range dependences, 
it is often of interest to concentrate martingales uniformly over time.

On the theoretical side, a fundamental limit to such concentration is expressed by the law of the iterated logarithm (LIL). 
However, this only concerns asymptotic behavior. 
In many applications, it is more natural to instead consider concentration that holds uniformly over all finite times.

This manuscript presents such bounds for the large classes of martingales 
which are addressed by Hoeffding \cite{H63} and Bernstein \cite{F75} inequalities. 
These new results are optimal within small constants, and can be viewed as finite-time generalizations of the upper half of the LIL.
 
To be concrete, 
the simplest nontrivial martingale for such purposes is the discrete-time random walk $\{M_t\}_{t=0,1,2,\dots}$ 
induced by flipping a fair coin repeatedly. 
It can be written as $M_t = \sum_{i=1}^t \sigma_i$, where $\sigma_i$ are i.i.d. Rademacher-distributed random variables  
($\prp{ \sigma_i = -1}{} = \prp{ \sigma_i = +1}{} = 1/2$), 
so we refer to it as the ``Rademacher random walk"; take $M_0 = 0$ w.l.o.g.

The LIL was first proved for the Rademacher random walk, by Khinchin:
\begin{thm}[Law of the iterated logarithm \cite{K24}]
\label{thm:lilorig}
Suppose $M_t$ is a Rademacher random walk.
Then with probability $1$,
$$ \limsup_{t \to \infty} \frac{\abs{M_t}}{\sqrt{t \log \log t}} = \sqrt{2} $$
\end{thm}

Our main concentration result for the Rademacher random walk generalizes this to hold over finite times.
\begin{thm}
\label{thm:newunifmart}
Suppose $M_t$ is a Rademacher random walk.
Then there is an absolute constant $C$ such that for any $ \delta < 1$, with probability $\geq 1 - \delta$, 
for \emph{all} $t \geq C \log \lrp{\frac{4}{\delta}} $ \emph{simultaneously}, 
the following are true: $\abs{M_t} \leq \frac{t}{e^2 \lrp{1 + \sqrt{1/3}}}$ and 
%\begin{align*}
$\abs{M_t} \leq \sqrt{ 3 t \lrp{ 2 \log \log \lrp{\frac{5 t}{ 2 \abs{M_t} }} + \log \left( \frac{2}{\delta} \right) }} $.
%\end{align*} 

(The latter implies $\abs{M_t} \leq \max \lrp{ \sqrt{ 3 t \lrp{ 2 \log \log \lrp{ \frac{5}{ 2} t} + \log \left( \frac{2}{\delta} \right) }} , 1 }$. 
It suffices if $C=173$.)
\end{thm}

Theorem \ref{thm:newunifmart} takes the form of the LIL upper bound as $t \to \infty$ for any fixed $\delta > 0$.
Interestingly, it also captures a finite-time tradeoff between $t$ and $\delta$.  
The $\log \log t$ term is dominated by the $\log \left( \frac{1}{\delta} \right)$ term for $\displaystyle t \lesssim e^{1 / \delta}$. 
In this regime, the bound is $\cO{\sqrt{ t \log \left( \frac{1}{\delta} \right) } }$, 
a time-uniform central limit theorem (CLT)-type bound below the LIL rate for small enough $t$ and $\delta$.
%This is of applied interest because $e^{1 / \delta}$ can often be extremely large, in which case the CLT regime can encompass all times realistically encountered in practice.

\subsection{Optimality of Theorem \ref{thm:newunifmart}}
We now show that Theorem \ref{thm:newunifmart} is optimal in a very strong sense. 

Suppose we are concerned with concentration of the random walk uniformly over time up to some fixed finite time $T$. 
If the failure probability $\delta \lesssim \frac{1}{\log T}$, 
then the $\log \frac{1}{\delta}$ term dominates the $\log \log t$ term for all $t < T$. 
In this case, the bound of Theorem \ref{thm:newunifmart} is $\cO{\sqrt{t \log \frac{1}{\delta}}}$ uniformly over $t < T$. 
This is optimal even for a fixed $t$ by binomial tail lower bounds.%, e.g. Slud's inequality.

The more interesting case for our purposes is when $\delta \gtrsim \frac{1}{\log T}$, 
in which case Theorem \ref{thm:newunifmart} gives a concentration rate of 
$\cO{ \sqrt{ t \log \log t + t \log \frac{1}{\delta} }}$, uniformly over $t < T$. 
As $T$ and $t$ increase without bound for any fixed $\delta > 0$, 
this rate becomes $\cO{\sqrt{t \log \log t}}$, 
which is unimprovable by the LIL. 

But the \emph{tradeoff} between $t$ and $\delta$ given in Theorem \ref{thm:newunifmart} 
is also essentially optimal, 
as shown by the following result.
\begin{thm}
\label{thm:lbinterestingregime}
There are absolute constants $C_1, C_2$ such that the following is true. 
Fix a finite time $T > C_2 \log \lrp{\frac{2}{\delta}}$, 
and fix any $\delta \in \left[ \frac{4}{ \log \lrp{ (T-1)/3 }} , \frac{1}{C_1} \right]$.
Then with probability \textbf{at most} $ 1 - \delta$, 
for all $t \in \left[C_2 \log \lrp{\frac{2}{\delta}}, T \right)$ simultaneously, $\abs{M_t} \leq \frac{2 }{ 3 e^2 } t$ and
$$ \abs{M_t} \leq \sqrt{ \frac{2}{3} t \left( \log \log \lrp{ \frac{ \frac{2}{3} t }{\abs{M_{t}} + 2 \sqrt{t / 3 } } } 
+ \log \lrp{\frac{1}{C_1 \delta}} \right)} $$
(It suffices if $C_1 = \lrp{\frac{420}{11}}^2$ and $C_2 = 164$.)
\end{thm}

%\begin{thm}
%\label{thm:lillbmain}
%There are absolute constants $C_1, C_2$ such that the following is true. 
%Fix a finite time $T > C_2 \log \lrp{\frac{4}{\delta}}$, 
%and take any $\delta \leq \frac{1}{C_1}$.
%Then with probability $\leq 1 - \delta$, 
%for all $t \in [C_2 \log \lrp{\frac{4}{\delta}}, T)$ simultaneously, $\abs{M_t} \leq \frac{2 }{3 e^2 } U_{t}$ and
%$$ \abs{M_t} \leq \sqrt{ \frac{2}{3} U_{t} \left( \log \log \lrp{ \frac{ \frac{2}{3} U_{t} }{\abs{M_{t}} + 2 \sqrt{ U_{t}/3 } } } + \log \lrp{\frac{1}{C_1 \delta}} \right)} $$
%(It suffices if $C_1 = \lrp{\frac{420}{11}}^2$ and $C_2 = 164$.)
%\end{thm}

There are no previous results in this vein, to our knowledge. 
So a principal contribution of this manuscript is to 
characterize the tradeoff between $t$ and $\delta$ in uniform concentration of measure over time, 
which is not addressed by the classical LIL of Theorem \ref{thm:lilorig}. 
Theorems \ref{thm:newunifmart} and \ref{thm:lbinterestingregime} constitute a sharp finite-time version of the LIL, 
analogous to how the Hoeffding bound for large deviations is a sharp finite-time version of the CLT's Gaussian tail for a fixed time.

\subsection{Paper Outline}
The proofs of both Theorems \ref{thm:newunifmart} and \ref{thm:lbinterestingregime} are possibly of independent interest, 
and the rest of this paper details and builds on them.

The proof of Theorem \ref{thm:newunifmart} 
extends the exponential moment method, the standard way of 
proving classical Chernoff-style bounds which hold for a fixed time. 
We use a technique, 
manipulating stopping times of a particular averaged supermartingale, 
which generalizes easily to many discrete- and continuous-time martingales, 
allowing us to prove iterated-logarithm concentration bounds for them as well. 
These martingale generalizations are given in Section \ref{sec:unifbernbounds}, 
with discussion of even more general settings like continuous time.
Proof details for these concentration bounds are in Section \ref{sec:pfmainthm}.

The proof of the anti-concentration bound of Theorem \ref{thm:lbinterestingregime} 
basically inverts the argument used to prove the concentration bounds. 
We give the details in Section \ref{sec:pfrrwlb}. 
Finally, Section \ref{sec:intgsupermartlem1} contains some ancillary results, 
and formalizes extensions discussed in the previous sections.

%-------------------------------------------------------------------------------------------------------------------------------------------------------------------------------------------------------------------------------------
%-------------------------------------------------------------------------------------------------------------------------------------------------------------------------------------------------------------------------------------

\section{Uniform Concentration Bounds for Martingales}
\label{sec:unifbernbounds}

In this section, we extend the random walk concentration result of Theorem \ref{thm:newunifmart} 
to broader classes of martingales. 
Some notation must be established first. 

We study the behavior of a real-valued stochastic process $M_t$ in a filtered probability space $(\Omega, \cF, \{\cF_t\}_{t \geq 0}, P)$, 
where $M_0 = 0$ w.l.o.g. 
For simplicity, only the discrete-time case $t \in \NN$ is considered hereafter; 
the results and proofs in this manuscript extend to continuous time as well (Remark \ref{remark:ctmarts}).
Define the difference sequence $\xi_t = M_t - M_{t-1}$ for all $t$ ($\xi_t$ being $\cF_t$-measurable),
and the cumulative conditional variance and quadratic variation: 
$V_t = \sum_{i=1}^t \evp{\xi_i^2 \mid \cF_{i-1}}{}$ and $Q_t = \sum_{i=1}^t \xi_i^2$ respectively.

Also recall the following standard definitions. 
A martingale $M_t$ (resp. supermartingale, submartingale) has $\evp{\xi_t \mid \cF_{t-1}}{} = 0$ (resp. $\leq 0, \;\geq 0$) for all $t$. 
A stopping time $\tau$ is a function on $\Omega$ such that $\{ \tau \leq t \} \in \cF_t \;\;\forall t$; 
notably, $\tau$ can be infinite with positive probability \cite{D10}.

\subsection{Uniform Second-Moment Martingale Concentration}
\label{sec:unifbern}

A few pertinent generalizations of Theorem \ref{thm:newunifmart} are now presented. 
The first is a direct iterated-logarithm analogue of Hoeffding's inequality \cite{H63} for martingales.

\begin{thm}[Uniform Hoeffding Bound]
\label{thm:unifah}
Let $M_t$ be a martingale, and 
suppose there are constants $\{c_i\}_{i \geq 1}$ such that for all $t \geq 1$, $\abs{M_t - M_{t-1}} \leq c_t$ w.p. 1. 
Fix any $ \delta < 1$ and define $\tau_0 = \min \left\{s : \sum_{i=1}^s c_i^2 \geq 173 \log \lrp{\frac{4}{\delta}} \right\}$.
Then with probability $\geq 1 - \delta$, 
for all $t \geq \tau_0$ simultaneously, 
$\abs{M_t} \leq \frac{\sum_{i=1}^t c_i^2}{e^2 \lrp{1 + \sqrt{1/3}}}$ and 
$$ \abs{M_t} \leq \sqrt{3 \lrp{\sum_{i=1}^t c_i^2} \lrp{ 2 \log \log \lrp{\frac{3 \lrp{\sum_{i=1}^t c_i^2}}{ 2 \abs{M_t} }} + \log \left( \frac{2}{\delta} \right) }} $$
\end{thm}

%\akshay{Practice!}
%Theorem Theorem \ref{thm:unifbmart} is particularly convenient for many applications  
%because the cumulative conditional variance $V_t$ marginalizes over the present time $t$, 
%and therefore can often be controlled usefully in practice.

%Another generalization of Theorem \ref{thm:newunifmart} is as follows.
%
%\begin{thm}
%\label{thm:unifqmart}
%Let $M_t$ be a martingale. 
%Fix any $ \delta < 1$ and define 
%%\newline
%$\tau_0 = \min \left\{s : \frac{1}{3} (2 V_s + Q_s) \geq 173 \log \lrp{\frac{4}{\delta}} \right\}$. 
%Then with probability $\geq 1 - \delta$, 
%for all $t \geq \tau_0$ simultaneously, 
%$\abs{M_t} \leq \frac{2 V_t + Q_t}{3 e^2 \lrp{1 + \sqrt{1/3}}}$ and 
%$$ \abs{M_t} \leq \sqrt{(2 V_t + Q_t) \lrp{ 2 \log \log \lrp{\frac{2 V_t + Q_t}{ 2 \abs{M_t} }} + \log \left( \frac{2}{\delta} \right) }} $$
%\end{thm}
%
%If each difference iterate is assumed a.s. bounded in Theorem \ref{thm:unifqmart}, 

A uniform counterpart to Bernstein's inequality can be derived similarly.

\begin{thm}[Uniform Bernstein Bound]
\label{thm:unifbmart}
Let $M_t$ be a martingale. 
Suppose the difference sequence is uniformly bounded: 
$\abs{M_t - M_{t-1}} \leq e^2$ w.p. $1$ for all $t \geq 1$. 
Fix any $ \delta < 1$ and define $\tau_0 = \min \left\{s : 2 (e-2) V_s \geq 173 \log \lrp{\frac{4}{\delta}} \right\}$.  
Then with probability $\geq 1 - \delta$, 
for all $t \geq \tau_0$ simultaneously, 
$\abs{M_t} \leq \frac{2 (e-2)}{e^2 \lrp{1 + \sqrt{1/3}}} V_t$ and 
$$ \abs{M_t} \leq \sqrt{6 (e-2) V_t \lrp{ 2 \log \log \lrp{\frac{3 (e-2) V_t}{ \abs{M_t} }} + \log \left( \frac{2}{\delta} \right) }} $$
\end{thm}

As with other Bernstein-type inequalities, the boundedness assumption on $\xi_t$ can be replaced by 
higher moment conditions (e.g. Lemma \ref{lem:bernubrelaxconstr} and the preceding discussion). 

The proofs of Theorems \ref{thm:unifah} and \ref{thm:unifbmart} are nearly identical to that of Theorem \ref{thm:newunifmart}. 
Further details on these topics are given in Section \ref{sec:mainpfgen}.

\subsection{Discussion}
\label{sec:extensions}

Some remarks on our results are in order. 

\begin{nremark}[Extension to Continuous Time]
\label{remark:ctmarts}
In many cases, our uniform results can be generalized to continuous-time martingales with an almost unchanged proof 
(e.g., this is true for the Wiener process $W_t$). 
Further explanation of this depends on the proof details, 
and therefore is deferred to Section \ref{sec:pfdiscuss}.
\qedinpw
\end{nremark}

\begin{nremark}[Extension to Super(sub)martingale Bounds]
\label{remark:doobdecomp}
%\normalfont
One-sided variants of Theorem \ref{thm:newunifmart} hold in many cases for super- (resp. sub-) martingales, 
giving a uniform upper (resp. lower) bound identical to that in Theorem \ref{thm:newunifmart}. 
When the Doob-Meyer decomposition (\cite{D10}) applies, 
%as is frequently the case, 
such bounds are immediate.
\qedinpw
\end{nremark}

\begin{nremark}[Extension to Sharpen Constants and Initial Time Conditions]
\label{remark:inittime}
\normalfont
The leading proportionality constant on the $\sqrt{t \log \log t}$ term in Theorem \ref{thm:newunifmart} is $\sqrt{6}$, 
clearly suboptimal in the limit $t \to \infty$ by the LIL. 
A modification to our proof lowers this constant arbitrarily close to its optimal value of $\sqrt{2}$ for high enough $t$, 
suggesting that our proof technique is quite tight. 
The tightened proofs generalize results of Robbins and Siegmund 
\cite{RS70, R70} to the non-asymptotic case, and to general martingales.
Section \ref{sec:constschainatoms} contains further details. 

The upper concentration bounds in this manuscript include an initial time condition $t \geq \tau_0$. 
% for $\tau_0 = \min \left\{ s: U_s \geq 173 \log \left( \frac{4}{\delta} \right) \right\}$, 
% where $U_t$ is a nondecreasing process measuring cumulative variance. 
For Theorem 2, it is straightforward to remove this condition without degrading the result: 
if $t < \tau_0$, 
an explicit union bound over fixed-time Hoeffding bounds immediately gives that 
$\abs{M_t} \leq \cO{\log \lrp{\frac{1}{\delta}}}$ w.h.p. for all $t < \tau_0$. 
Combining this with Theorem 2 gives a uniform bound, over all times $t \geq 0$, of 
$\abs{M_t} \leq \cO{ \sqrt{ t \lrp{ \log \log t + \log \frac{1}{\delta} }} + \log \lrp{\frac{1}{\delta}} }$. 
The $t$-independent additive $\log \lrp{\frac{1}{\delta}} $ term matches 
standard Bernstein/Bennett concentration inequalities \cite{BLM13} for a fixed time. 

To extend this argument to the upper bounds of Theorems \ref{thm:unifah} and \ref{thm:unifbmart}, 
such an explicit union bound is no longer usable. 
But our proof techniques using stopping times extend naturally to these cases -- see Section \ref{sec:inittime}.
\qedinpw
\end{nremark}

%\begin{nremark}[Relation to Other Maximal Bounds] 
When considering uniform martingale concentration over all times without an explicit union bound, 
the basic tools are Doob's maximal inequality for nonnegative supermartingales (\cite{D10}, Exercise 5.7.1), 
Hoeffding's maximal inequality \cite{H63}, 
and Freedman's Bernstein-type inequality \cite{F75}. 
These can all be easily proved with the techniques of this manuscript 
(similar to the proof of Theorem \ref{thm:varunifub}).
However, the latter two results are fundamentally weaker than ours, 
in that they only hold uniformly over a finite time interval, 
and degrade to triviality as the interval grows infinite. 
%\akshay{Mention the $\log t$ vs. $\log \log V_t$ is about the variance term.}
%\qedinpw
%\end{nremark}

Our uniform Hoeffding/Bernstein-type bounds in Section \ref{sec:unifbern} 
achieve optimal rates in the variance and $\delta$ parameters as well. 
These generalize martingale LILs like the classic result of Stout \cite{S70}, 
which for large classes of martingales makes a statement similar to Theorem \ref{thm:lilorig}, 
except concerning the ratio $\abs{M_t} / \sqrt{V_t \log \log V_t}$. 
The finite-time upper Hoeffding-bound LIL can be proved with an epoch-based approach \cite{JMNB14} standard in proofs of the (asymptotic) LIL \cite{D10}. 
Our technique can be viewed as generalizing that idea using stopping time manipulations.

Theorem \ref{thm:newunifmart}'s tradeoff between $t$ and $\delta$ 
describes some of the interplay between the CLT and the LIL when uniform bounds are taken of partial sums of suitable i.i.d. variables. 
A similar question has been explored with a different statistical emphasis by Darling and Erd\H{o}s \cite{DE56} and subsequent work, 
though only as $t \to \infty$ to our knowledge.

%-------------------------------------------------------------------------------------------------------------------------------------------------------------------------------------------------------------------------------------

%\akshay{Add mention of tests to ask YP about: 
%EFKP test, 
%Dvoretzky-Erdos test for liminf of BM,
%}

%-------------------------------------------------------------------------------------------------------------------------------------------------------------------------------------------------------------------------------------
%-------------------------------------------------------------------------------------------------------------------------------------------------------------------------------------------------------------------------------------
%-------------------------------------------------------------------------------------------------------------------------------------------------------------------------------------------------------------------------------------

\section{Proof of Theorem \ref{thm:newunifmart} and Concentration Bounds}
\label{sec:pfmainthm}
Define the (deterministic) process $U_t = t$ 
(a notational convenience, to ease extension of this proof to the martingale case discussed in Section \ref{sec:unifbernbounds}).
Also define $k := \frac{1}{3}$ and $\lambda_0 := \lrp{e^2 \lrp{1 + \sqrt{k}}}^{-1}$.

In this section, we prove the following bound, which is a slightly more precise version of Theorem \ref{thm:newunifmart}:
\begin{thm}
\label{thm:genunifmart}
Let $M_t$ be a Rademacher random walk. 
Fix any $ \delta < 1$ and define the time $\tau_0 = \min \left\{s : U_s \geq \frac{2}{\lambda_0^2} \log \lrp{\frac{4}{\delta}} \right\} $.
Then with probability $\geq 1 - \delta$, 
for all $t \geq \tau_0$ simultaneously, 
$\abs{M_t} \leq \lambda_0 U_t$ and 
$$ \abs{M_t} \leq \sqrt{ \frac{2 U_{t}}{1-k} \log \left( \frac{ 2 \log^2 \lrp{\frac{U_{t}}{ \lrp{1 - \sqrt{k}} \abs{ M_{t}} } } }{\delta} \right)}  $$
\end{thm}

The proof invokes the Optional Stopping Theorem. %in martingale theory,
\begin{thm}[Optional Stopping for Nonnegative Supermartingales (\cite{D10}, Theorem 5.7.6)]
\label{thm:optstoppingsupermart}
Let $M_t$ be a nonnegative supermartingale. Then if $\tau$ is a (possibly infinite)
stopping time, $\evp{M_\tau}{} \leq \evp{M_0}{}$.
\end{thm}
The version we use here explicitly exploits the favorable convergence properties of nonnegative supermartingales 
when the stopping time is infinite.

Our argument begins by appealing to a standard exponential supermartingale construction.

\begin{lem}
\label{lem:supermartconstr}
The process $X_t^\lambda := \expp{ \lambda M_t - \frac{\lambda^2}{2} U_t}$ is a supermartingale for any $\lambda \in \RR$.
\end{lem}
\begin{proof}
Using Hoeffding's Lemma, for any $\lambda \in \RR$ and $t \geq 1$, 
$\evp{\expp{\lambda \xi_t } \mid \cF_{t-1}}{} \leq \expp{\frac{\lambda^2}{8} (2^2)} = \expp{\frac{\lambda^2}{2} }$. 
Therefore, $\evp{\expp{\lambda \xi_t - \frac{\lambda^2}{2} } \mid \cF_{t-1}}{} \leq 1$, so $\evp{X_t^\lambda \mid \cF_{t-1}}{} \leq X_{t-1}^\lambda$.
\end{proof}

The result is derived through various manipulations of this supermartingale $X_t^\lambda$. 

For the rest of the proof, for all $t$, assume that $M_t \neq 0$.
This is with full generality, because when $M_t = 0$, the bound of Theorem \ref{thm:newunifmart} trivially holds.

\subsection{A Time-Uniform Law of Large Numbers}
\label{sec:btstrp}

The desired result, Theorem \ref{thm:genunifmart}, 
uniformly controls $ \abs{M_t} / \sqrt{U_t \log \log U_t}$, 
but we first control $ \abs{M_t} / U_t$. 
This generalizes the (strong) law of large numbers ([S]LLN), for any failure probability $\delta > 0$, 
uniformly over finite times. 

While a weaker result than Theorem \ref{thm:genunifmart}, 
this concisely demonstrates our principal proof techniques, 
and is independently necessary as a ``bootstrap" for the main bound.

The first step is to establish a moment bound which holds at any stopping time, 
by averaging supermartingales from the family $\left\{ \expp{ \lambda M_t - \frac{\lambda^2}{2} U_t} \right\}_{\lambda \in \RR}$ 
using a particular weighting over $\lambda$.

\begin{lem}
\label{lem:btstrpmgf}
For any stopping time $\tau$, 
$ \evp { \expp{ \lambda_0 \abs{M_\tau} - \frac{\lambda_0^2}{2} U_\tau} }{} \leq 2 $.
\end{lem}
\begin{proof}
Recall the definition of $X_t^\lambda$ from Lemma \ref{lem:supermartconstr}. 
Here we set the free parameter $\lambda$ in the process $X_t^\lambda$ to get a process $Y_t$. 
$\lambda$ is set stochastically: $\lambda \in \{ - \lambda_0, \lambda_0 \}$ with probability $\frac{1}{2}$ each. 
After marginalizing over $\lambda$, the resulting process is
\begin{align}
\label{btstrpmgf}
Y_t &= \frac{1}{2} \expp{ \lambda_0 M_t - \frac{\lambda_0^2}{2} U_t} + \frac{1}{2} \expp{ - \lambda_0 M_t - \frac{\lambda_0^2}{2} U_t} \nonumber \\
&\geq \frac{1}{2} \expp{ \lambda_0 \abs{M_t} - \frac{\lambda_0^2}{2} U_t}
\end{align}
Now take $\tau$ to be any stopping time as in the lemma statement. 
Then $\evp{\expp{ \lambda_0 M_\tau - \frac{\lambda_0^2}{2} U_\tau} }{} = \evp{X_\tau^{\lambda = \lambda_0} }{} \leq 1$, 
where the inequality is by the Optional Stopping Theorem (Theorem \ref{thm:optstoppingsupermart}). 
Similarly, $\evp{X_\tau^{\lambda = -\lambda_0} }{} \leq 1$.

So $\evp{Y_\tau}{} = \frac{1}{2} \lrp{ \evp{ X_\tau^{\lambda = -\lambda_0}}{} + \evp{X_\tau^{\lambda = \lambda_0}}{} } \leq 1$. 
Combining this with \eqref{btstrpmgf} gives the result.
\end{proof}

A particular setting of $\tau$ extracts the desired uniform LLN bound from Lemma \ref{lem:btstrpmgf}.

\begin{thm}
\label{thm:varunifub}
Fix any $\delta > 0$. With probability $\geq 1-\delta$, 
for all $t \geq \min \left\{t : U_t \geq \frac{2}{\lambda_0^2} \log \lrp{\frac{2}{\delta}} \right\} $ simultaneously, 
$$ \frac{\abs{M_t}}{U_t} \leq \lambda_0 $$
\end{thm}

\begin{proof}
For convenience define $\tau_1 = \min \left\{t : U_t \geq \frac{2}{\lambda_0^2} \log \lrp{\frac{2}{\delta}} \right\}$.
Define the stopping time 
$ \tau = \min \left\{ t \geq \tau_1 : \frac{\abs{M_t}}{U_t} > \lambda_0 \right\} $.
Then it suffices to prove that $P(\tau < \infty) \leq \delta$.

On the event $\{ \tau < \infty \}$, we have $\frac{\abs{M_\tau}}{U_\tau} > \lambda_0$ by definition of $\tau$. 
Therefore, using Lemma \ref{lem:btstrpmgf},
\begin{align*}
2 &\geq \evp { \expp{ \lambda_0 \abs{M_\tau} - \frac{\lambda_0^2}{2} U_\tau} }{} 
\geq \evp { \expp{ \lambda_0 \abs{M_\tau} - \frac{\lambda_0^2}{2} U_\tau} \mid \tau < \infty}{} P(\tau < \infty) \\
&\stackrel{(a)}{>} \evp { \expp{ \lambda_0^2 U_\tau - \frac{\lambda_0^2}{2} U_\tau } }{} P(\tau < \infty) 
\stackrel{(b)}{\geq} \frac{2}{\delta} P(\tau < \infty)
\end{align*}
where $(a)$ uses that $\frac{\abs{M_\tau}}{U_\tau} > \lambda_0$ when $\tau < \infty$, 
and $(b)$ uses $U_\tau \geq U_{\tau_1} \geq \frac{2}{\lambda_0^2} \log \lrp{\frac{2}{\delta}}$. 
Therefore, $P(\tau < \infty) \leq \delta$, as desired. 
\end{proof}

The process $U_t$ is increasing in any case of interest, 
implying that $\abs{M_t} / U_t \leq \lambda_0$ uniformly in $t$ after some finite initial time. 
The setting of $\lambda_0$ happens to fit with the rest of our main proof, 
but this choice of $\lambda_0$ in Theorem \ref{thm:varunifub} is arbitrary. 
The same proof method in fact defines a family of bounds parametrized by $\lambda_0$; 
collectively, these express the SLLN for finite times.

%--------------------------------------------------------------------------------------------------------------------------------------------------------------------------------------------------------------------------------------

\subsection{Proof of Theorem \ref{thm:genunifmart}}
\label{sec:mainpf}

We proceed to prove Theorem \ref{thm:genunifmart}, using the SLLN bound of Theorem \ref{thm:varunifub} and its proof techniques.

\subsubsection{Preliminaries}
\label{sec:prelimadelta}

A little further notation is required for the rest of the proof. 

For any event $E \subseteq \Omega$ of nonzero measure, let $\evp{\cdot}{E}$ denote the expectation restricted to $E$, 
i.e. $\evp{f}{E} = \frac{1}{P(E)} \int_E f(\omega) P(d \omega)$ for a measurable function $f$ on $\Omega$. 
Similarly, dub the associated measure $P_E$, where for any event $\Xi \subseteq \Omega$ we have $P_E(\Xi) = \frac{P(E \cap \Xi)}{P(E)}$.

Consider the ``good" event of Theorem \ref{thm:varunifub}, in which its uniform deviation bound holds w.p. $\geq 1 - \delta$ for some $\delta$; 
call this event $A_\delta$. 
Formally, 
\begin{align}
\label{defofA}
A_\delta = 
\left\{ \omega \in \Omega : \frac{\abs{M_t}}{U_t} \leq \lambda_0 \;\;
\forall t \geq \min \left\{s : U_s \geq \frac{2}{\lambda_0^2} \log \lrp{\frac{2}{\delta}} \right\} \right\} 
\end{align}
Theorem \ref{thm:varunifub} states that $P(A_\delta) \geq 1-\delta$.

It will be necessary to shift sample spaces from $A_\delta$ to $\Omega$. 
The shift should be small in measure because $P(A_\delta) \geq 1 - \delta$; this is captured by the following simple observation.
\begin{lem}
\label{lem:shiftoutofA}
Define $A_\delta$ as in \eqref{defofA}. For any nonnegative random variable $X$ on $\Omega$,
\begin{align*}
\evp{X}{A_\delta} \leq \frac{1}{1-\delta} \evp{X}{}
\end{align*}
\end{lem}
\begin{proof}
Since $X \geq 0$, using Theorem \ref{thm:varunifub}, $\evp{X}{} = \evp{X}{A_\delta} P(A_\delta) + \evp{X}{A_\delta^c} P(A_\delta^c) \geq \evp{X}{A_\delta} (1-\delta)$.
\end{proof}

Define $X_t^\lambda$ as in Lemma \ref{lem:supermartconstr}. 
The idea of the proof is to choose $\lambda$ stochastically from a probability space $(\Omega_\lambda, \cF_\lambda, P_\lambda)$. 
The parameter $\lambda$ is chosen independently of $\xi_1, \xi_2, \dots$, 
so that $X_t^\lambda$ is defined on the product space.
Write $\mathbb{E}^{\lambda} \left[ \cdot \right]$ to denote the expectation with respect to $(\Omega_\lambda, \cF_\lambda, P_\lambda)$. 

To be consistent with previous notation, 
we continue to write $\evp{\cdot}{}$ to denote the expectation w.r.t. the original probability space $(\Omega, \cF, P)$ 
which encodes the stochasticity of $M_t$. 
As mentioned earlier, we use subscripts for expectations conditioned on events in this space, e.g. $\evp{X}{A_\delta}$. 
As an example, $\evp{\cdot}{\Omega} = \evp{\cdot}{}$.

\subsubsection{Proof of Theorem \ref{thm:genunifmart}}
\label{sec:subfinalpf}

The main result can now be proved. 
The first step is to choose $\lambda$ stochastically in the supermartingale $X_t^\lambda$ 
and bound the effect of averaging over $\lambda$
(analogous to Lemma \ref{lem:btstrpmgf} in the proof of the bootstrap bound). 
\begin{lem}
\label{lem:intgsupermartlem1}
Define $\tau_0$ as in Theorem \ref{thm:genunifmart}, 
and $A_\delta$ as in \eqref{defofA} for any $\delta$. 
Then for any stopping time $\tau \geq \tau_0$,
$$ \evp{\mathbb{E}^{\lambda} \left[ X_{\tau}^\lambda \right]}{A_{\delta}} 
\geq \evp{ \frac{ 2 \expp{ \frac{M_{\tau}^2}{2 U_{\tau}} (1-k) } } { \log^2 \lrp{\frac{U_{\tau}}{ \lrp{1 - \sqrt{k}} \abs{ M_{\tau}} } } } }{A_\delta} $$
\end{lem}
The proof relies on estimating an integral and is deferred to Section \ref{sec:intgsupermartlem1}.

Lemma \ref{lem:intgsupermartlem1} can be converted into the desired uniform bound using a particular choice of stopping time, 
analogously to how the bootstrap bound Theorem \ref{thm:varunifub} is derived from Lemma \ref{lem:btstrpmgf}. 
However, this time a shift in sample spaces is also needed to yield Theorem \ref{thm:genunifmart}, 
since Lemma \ref{lem:intgsupermartlem1} uses $A_\delta$ instead of $\Omega$.

\begin{proof}[Proof of Theorem \ref{thm:genunifmart}]
Define the stopping time 
\begin{align}
\label{deftauub}
\displaystyle \tau = \min \Bigg\{ t \geq &\tau_0 \colon 
\abs{M_{t}} > \lambda_0 U_t \;\;\lor\;\; \nonumber \\
&\left( \abs{M_{t}} \leq \lambda_0 U_t \;\land\; 
\abs{M_{t}} > \sqrt{ \frac{2 U_t}{1-k} \log \left( \frac{ 2 \log^2 \lrp{\frac{U_{t}}{ \lrp{1 - \sqrt{k}} \abs{ M_{t}} } } }{\delta} \right)}  \right) \Bigg\}
\end{align}
It suffices to prove that $P (\tau = \infty) \geq 1 - \delta$.
On the event $\left\{ \{ \tau < \infty \} \cap A_{\delta/2} \right\}$, we have 
\begin{align*}
\abs{M_\tau} > \sqrt{ \frac{2 U_\tau}{1-k} \log \left( \frac{ 2 \log^2 \lrp{\frac{U_{\tau}}{ \lrp{1 - \sqrt{k}} \abs{ M_{\tau}} } } }{\delta} \right)} 
&\iff \frac{ 2 \expp{ \frac{M_{\tau}^2}{2 U_{\tau}} (1-k) } } { \log^2 \lrp{\frac{U_{\tau}}{ \lrp{1 - \sqrt{k}} \abs{ M_{\tau}} } } } > \frac{4}{\delta}
\end{align*} 
Therefore, using Lemma \ref{lem:intgsupermartlem1} 
and the nonnegativity of 
$\frac{ 2 \expp{ \frac{M_t^2}{2 U_t} (1-k) } } { \log^2 \lrp{\frac{U_t}{ \lrp{1 - \sqrt{k}} \abs{ M_t} } } }$ on $A_{\delta/2}$, we have 
\begin{align*}
2 &\geq \frac{1}{1-\frac{\delta}{2}} = \frac{\mathbb{E}^{\lambda} \left[ \evp{X_0^\lambda}{} \right] }{1-\frac{\delta}{2}} 
\stackrel{(a)}{\geq} \frac{ \mathbb{E}^{\lambda} \left[ \evp{X_{\tau}^\lambda}{} \right] }{1-\frac{\delta}{2}} 
\stackrel{(b)}{\geq} \mathbb{E}^{\lambda} \left[ \evp{X_\tau^\lambda}{A_{\delta/2}} \right] \\
&\stackrel{(c)}{=} \evp{\mathbb{E}^{\lambda} \left[ X_\tau^\lambda \right]}{A_{\delta/2}} 
\stackrel{(d)}{\geq} \evp{ \frac{ 2 \expp{ \frac{M_{\tau}^2}{2 U_{\tau}} (1-k) } } { \log^2 \lrp{\frac{U_{\tau}}{ \lrp{1 - \sqrt{k}} \abs{ M_{\tau}} } } } }{A_{\delta/2}} \\
&\geq \evp{ \frac{ 2 \expp{ \frac{M_{\tau}^2}{2 U_{\tau}} (1-k) } } { \log^2 \lrp{\frac{U_{\tau}}{ \lrp{1 - \sqrt{k}} \abs{ M_{\tau}} } } } \mid \tau < \infty}{A_{\delta/2}} 
P_{A_{\delta/2}}(\tau < \infty) 
\stackrel{(e)}{>} \frac{4}{\delta} P_{A_{\delta/2}}(\tau < \infty)
\end{align*}
where $(a)$ is by Optional Stopping 
(Theorem \ref{thm:optstoppingsupermart}; note that $\tau$ can be unbounded), 
$(b)$ is by Lemma \ref{lem:shiftoutofA}, 
$(c)$ is by Tonelli's Theorem, 
$(d)$ is by Lemma \ref{lem:intgsupermartlem1}, 
and $(e)$ is by the definitions of $A_{\delta/2}$ in \eqref{defofA} 
and of $\tau$ in \eqref{deftauub}.

After simplification, this gives 
\begin{align}
\label{ubstgoodeventprob}
P_{A_{\delta/2}} (\tau < \infty) \leq \delta/2 \implies P_{A_{\delta/2}} (\tau = \infty) \geq 1 - \frac{\delta}{2}
\end{align} 
Therefore,
\begin{align*}
P (\tau = \infty) &\geq P (\left\{ \tau = \infty \right\} \cap A_{\delta/2}) \stackrel{(a)}{=} P_{A_{\delta/2}} (\tau = \infty) P (A_{\delta/2}) \\
&\stackrel{(b)}{\geq} \lrp{1-\frac{\delta}{2}} \lrp{1-\frac{\delta}{2}} \geq 1-\delta
\end{align*}
where $(a)$ uses the definition of $P_{A_{\delta/2}} (\cdot)$ and $(b)$ uses \eqref{ubstgoodeventprob} and Theorem \ref{thm:varunifub}. 
The result follows.
\end{proof}

\subsection{Proof Discussion}
\label{sec:pfdiscuss}

%The proof of Lemma \ref{lem:intgsupermartlem1} (particularly Eq. \eqref{laplacelbeqn}) 
%can be viewed as a simplified version of Laplace's method for approximating exponential integrals (\cite{B06}, Sec. 4.4); 
%the proof takes a rectangular lower bound of the area under the Gaussian-like ``peak" of the exponential integrand. 

Most of the tools used in this proof, particularly optional stopping as in Theorem \ref{thm:optstoppingsupermart}, 
extend seamlessly to the continuous-time case. 
The main potential obstacle to this is in the first step -- 
establishing an exponential supermartingale construction of the form of Lemma \ref{lem:supermartconstr}. 
This is easily done in many situations of interest, 
as demonstrated by the archetypal result that the standard geometric Brownian motion 
$\expp{\lambda W_t - \frac{\lambda^2}{2} t}$ is precisely a martingale for any $\lambda \in \RR$, 
where $W_t$ is the standard Wiener process. 

A direct antecedent to this manuscript is a pioneering line of work by Robbins and colleagues \cite{DR67, RS70, R70} 
that investigates the powerful method of averaging martingales. 
%Along with its sequels \cite{L76}.
For the most part, it only considers the asymptotic regime, 
though Darling and Robbins \cite{DR67} do briefly treat finite times (with far weaker $\delta$ dependence).
More recently, de la Pe\~{n}a et al. (\cite{DLPKL07} and references therein) revisit their techniques with a different emphasis and normalization for $M_t$.

The idea of using stopping times in the context of uniform martingale concentration 
goes back at least to work of Robbins \cite{R70} with Siegmund \cite{RS70}, 
and was then notably used by Freedman \cite{F75}.
%The averaging and stopping time techniques have been combined in a very specific context \cite{AYPS12} 
%(subsumed by our results), 
%but not for martingale bounds in any general setting. 

Our proof techniques are conceptually related to ideas from Shafer and Vovk (\cite{SV01}, Ch. 5), 
who describe how to view the LIL as emerging from a game. 
Departing from traditional approaches, 
they motivate the exponential supermartingale construction directly using Taylor expansions
%\footnote{This can otherwise be motivated with the continuous-time case, 
%where it is an exact martingale due to CLT effects in the Donsker continuous-time limit.}
and prove the (asymptotic) LIL by averaging such supermartingales. 
% Another use of our mixture distribution for a related purpose is in \cite{SSVV11}.

Two final interesting but unexplored connections bear mentioning.
First, our proof technique incorporates relative variation ($\abs{M_t} / U_t$) at multiple scales at once, 
possible here because the index set (time) is totally ordered. 
There is an analogy to well-developed general chaining techniques \cite{T05} 
that have been used to great effect to uniformly bound processes indexed on metric spaces, 
by using covering arguments which are also sensitive to variation at different scales. 
Second, it has been previously noted \cite{R02, HK04} that martingale stopping times and uniform concentration are dual, 
and our technique provides another connection between them.

%-------------------------------------------------------------------------------------------------------------------------------------------------------------------------------------------------------------------------------------
%-------------------------------------------------------------------------------------------------------------------------------------------------------------------------------------------------------------------------------------

\section{Proof of Theorem \ref{thm:lbinterestingregime} (Anti-Concentration Bound)}
\label{sec:pfrrwlb}

In this section, let $M_t$ be the Rademacher random walk, $k := \frac{1}{3}$, 
and $C_1$ be as defined in Theorem \ref{thm:lbinterestingregime}. 
Our proof will use submartingales rather than supermartingales; 
therefore, we will employ a standard optional stopping theorem for submartingales (paralleling Theorem \ref{thm:optstoppingsupermart}):
\begin{thm}[Optional Stopping for Submartingales (\cite{D10})]
\label{thm:optstoppingsubmart}
Let $M_t$ be a submartingale. Then if $\tau$ is an a.s. bounded stopping time, $\evp{M_\tau}{} \geq \evp{M_0}{}$.
\end{thm}

We will also construct a family of exponential submartingales, 
analogous to the supermartingale construction of Lemma \ref{lem:supermartconstr} in the concentration bound proof.
\begin{lem}
\label{lem:submartconstr}
The process $Z_t^\lambda := \expp{ \lambda M_t - k \lambda^2 U_t}$ is a submartingale for $\lambda \in \left[- \frac{1}{e^2}, \frac{1}{e^2} \right]$.
\end{lem}
\begin{proof}
We rely on the inequality $\cosh(x) \geq e^{k x^2}$ over $x \in \left[- \frac{1}{e^2}, \frac{1}{e^2} \right]$, so that
$\evp{Z_t^\lambda \mid \cF_{t-1}}{}
= \evp{\expp{ \lambda \xi_t} \mid \cF_{t-1}}{} e^{- k \lambda^2} Z_{t-1}^\lambda
= \cosh(\lambda) e^{- k \lambda^2} Z_{t-1}^\lambda \geq Z_{t-1}^\lambda$.
\end{proof}

%-------------------------------------------------------------------------------------------------------------------------------------------------------------------------------------------------------------------------------------

\subsection{Preliminaries and Proof Overview}
\label{sec:lillbprelims}

%\akshay{We have two methods (the LLN and the "which envelope determines the stopping time" one) which both hold for minimum times of $\log 1/\delta$. 
%Make a note of this.}

Many aspects of this proof parallel that of the concentration bound of Theorem \ref{thm:genunifmart}.
Again, the idea is to choose $\lambda$ stochastically from a probability space $(\Omega_\lambda, \cF_\lambda, P_\lambda)$ 
such that $ P_\lambda (d \lambda) = \frac{d \lambda}{\abs{\lambda} \lrp{ \log \frac{1}{\abs{\lambda}} }^2} $; 
and the parameter $\lambda$ is chosen independently of the $\xi_1, \xi_2, \dots$, 
so that $Z_t^\lambda$ is defined on the product space.

As in the concentration bound proof, write $\mathbb{E}^{\lambda} \left[ \cdot \right]$ 
to denote the expectation with respect to $(\Omega_\lambda, \cF_\lambda, P_\lambda)$. 
For consistency with previous notation, 
we continue to write $\evp{\cdot}{}$ to denote the expectation w.r.t. the original probability space $(\Omega, \cF, P)$ 
which encodes the stochasticity of $M_t$. 

Just as for the concentration bound, our proof of this anti-concentration bound reasons about the value of a particular stopping time. 
However, the stopping time used here is slightly different, 
because the conditions for convergence of our submartingales are more restrictive than those for supermartingales. 

To be concrete, define $\sigma_{\delta} := \frac{e^4}{k} \log \lrp{ \frac{2}{\delta} }$. 
For a given finite time horizon $T > \sigma_{\delta}$, 
we use a stopping time defined as
\begin{align}
\label{stdeflb}
\displaystyle \tau (T) := \min \Bigg[ &\min \bigg\{ t \in [\sigma_{\delta}, T) \colon 
\abs{M_{t}} > \frac{2k U_{t}}{e^2 } \;\;\lor\;\; \nonumber \\
&\left[ \abs{M_{t}} \leq \frac{2k U_{t}}{e^2 } \;\land\; 
\abs{M_{t}} > \sqrt{ 2k U_{t} \log \left( \frac{ \log \lrp{ \frac{2k U_{t} }{\abs{M_{t}} + 2 \sqrt{k U_{t}} } } }{C_1 \delta} \right)}  \right] \Bigg\} , \; T \Bigg]
\end{align}

We also require a moment bound for the $\lambda$-mixed submartingales, whose proof is in Section \ref{sec:pfoflem:lbintestimate}. 
\begin{lem}
\label{lem:lbintestimate}
For any $t$, 
$\mathbb{E}^{\lambda} \left[ Z_t^\lambda \right] \leq G_{t} :=
\begin{cases}
\frac{ 15 \expp{ \frac{M_t^2}{4k U_t} }}{ \log \lrp{ \frac{2k U_t}{\abs{M_t} + 2 \sqrt{k U_t} } } }  &\;\;  \mbox{if } \abs{M_t} \leq \frac{2k}{e^2} U_t \\
\frac{ 7 \expp{ \frac{M_t^2}{4k U_t} }}{ \log \lrp{ \frac{2k U_t}{ 2 \sqrt{k U_t} } } } &\;\;  \mbox{if } \abs{M_t} > \frac{2k}{e^2} U_t
\end{cases}$
\end{lem}

Finally, we use a ``one-sided Lipschitz" characterization of $M_t$ -- that $G_t$ will not grow too fast when $t \approx \tau(T)$:
\begin{lem}
\label{lem:cannotincreasemuch}
For any $T > \frac{e^4}{k} \log 2$, $ G_{\tau(T)} \leq \frac{14}{11} G_{\tau(T) - 1} $.
\end{lem}

With these tools, the proof can be outlined.
\begin{proof}[Proof Sketch of Theorem \ref{thm:lbinterestingregime}]
Here fix $T$ and write $\tau = \tau(T)$ as defined in \eqref{stdeflb}.
It suffices to prove that $P (\tau < T) \geq \delta$ under the given assumption on $\delta$.
We have
\begin{align}
\label{eq:mastereqlb}
1 &\stackrel{(a)}{\leq} \mathbb{E}^{\lambda} \left[ \evp{Z_\tau^\lambda}{} \right] 
\stackrel{(b)}{=} \evp{\mathbb{E}^{\lambda} \left[ Z_\tau^\lambda \right]}{} \nonumber \\
&\stackrel{(c)}{\leq} \evp{ G_{\tau} }{} \leq \evp{ G_{\tau} \mid \tau < T }{} P_{} (\tau < T) + \evp{ G_{\tau} \mid \tau = T }{} \nonumber \\
&\stackrel{(d)}{\leq} \frac{14}{11} \lrp{ \evp{ G_{\tau-1} \mid \tau < T }{} P_{} (\tau < T) + \evp{ G_{\tau-1} \mid \tau = T }{} }
\end{align}
where $(a)$ is by Optional Stopping (Theorem \ref{thm:optstoppingsubmart}), 
$(b)$ is by Tonelli's Theorem,
$(c)$ is by Lemma \ref{lem:lbintestimate}, 
and $(d)$ is by Lemma \ref{lem:cannotincreasemuch}. 
The result is then proved by upper-bounding $\evp{ G_{\tau-1} \mid \tau < T }{}$ and $\evp{ G_{\tau-1} \mid \tau = T }{}$, 
using the upper bounds on $\abs{M_{\tau-1}}$ given by the definition of $\tau$.
\end{proof}

\subsection{Full Proof of Theorem \ref{thm:lbinterestingregime}}
\label{sec:lillbpf}

\begin{proof}[Proof of Theorem \ref{thm:lbinterestingregime}]
Throughout this proof, fix $T$ and write $\tau = \tau(T)$ as defined in \eqref{stdeflb}. 
Also let $C_1, C_2$ be the absolute constants defined in the theorem statement. 
It suffices to prove that $P (\tau < T) \geq \delta$ under the given assumption on $\delta$. 

We do this by working with \eqref{eq:mastereqlb} from the proof sketch; this states that 
\begin{align}
\label{eq:mastereqlbredux}
\frac{11}{14} \leq \evp{ G_{\tau-1} \mid \tau < T }{} P_{} (\tau < T) + \evp{ G_{\tau-1} \mid \tau = T }{}
\end{align}

By definition of $\tau$, we have
\begin{align}
\label{endstconditionlbs}
\abs{M_{\tau-1}} \leq \frac{2k }{e^2 } U_{\tau-1}
\qquad \mbox{and} \qquad
\abs{M_{\tau-1}} \leq \sqrt{ 2k U_{\tau-1} \log \left( \frac{ \log \lrp{ \frac{2k U_{\tau-1} }{\abs{M_{\tau-1}} + 2 \sqrt{k U_{\tau-1}} } } }{C_1 \delta} \right)}
\end{align}
Therefore, substituting the definition in \eqref{endstconditionlbs} into the definition of $G_{\tau-1}$,
\begin{align}
\label{controlGub}
G_{\tau-1} = \frac{ 15 \expp{ \frac{M_{\tau-1}^2}{4k U_{\tau-1}} }}{ \log \lrp{ \frac{2k U_{\tau-1}}{\abs{M_{\tau-1}} + 2 \sqrt{k U_{\tau-1}} } } }
\leq \frac{15}{ \sqrt{ C_1 \delta \log \lrp{ \frac{2k U_{\tau-1}}{\abs{M_{\tau-1}} + 2 \sqrt{k U_{\tau-1}} } } }} 
\end{align}

So from \eqref{controlGub},
\begin{align}
\label{contribminimalinenvelope}
\evp{ G_{\tau-1} \mid \tau = T }{} 
\leq \evp{ \frac{15}{ \sqrt{ C_1 \delta \log \lrp{ \frac{2k U_{T-1}}{\abs{M_{T-1}} + 2 \sqrt{k U_{T-1}} } } }} \mid \tau = T }{} 
\leq \frac{15}{\sqrt{C_1}} = \frac{11}{28}
\end{align}
where the last inequality is by the assumption $\delta \geq \frac{4}{\log \lrp{ k U_{T-1} }}$ and Lemma \ref{lem:nostlbdelta}.

Also, from \eqref{controlGub},
\begin{align*}
\evp{ G_{\tau-1} \mid \tau < T }{} 
&\leq \evp{  \frac{15}{ \sqrt{C_1 \delta \log \lrp{ \frac{2k U_{\tau-1}}{\abs{M_{\tau-1}} + 2 \sqrt{k U_{\tau-1}} } } }} \mid \tau < T }{} \\
&\stackrel{(a)}{\leq} \evp{ \frac{15}{ \sqrt{C_1 \delta \log \lrp{ \frac{2k U_{\tau-1}}{ 2k e^{-2} U_{\tau-1} + 2 \sqrt{k U_{\tau-1}} } } }}  \mid \tau < T }{}  \\
&= \evp{ \frac{15}{ \sqrt{ C_1 \delta \log \lrp{ \lrp{ \frac{1 }{e^2 } + \frac{1}{\sqrt{k U_{\tau-1}}} }^{-1} } }} \mid \tau < T }{}  
\stackrel{(b)}{\leq} \frac{15}{ \sqrt{C_1 \delta}} 
\leq \frac{15}{\delta \sqrt{C_1}} = \frac{11}{28 \delta}
\end{align*}
where $(a)$ uses $\abs{M_{\tau-1}} \leq \frac{2k }{e^2 } U_{\tau-1}$ (by \eqref{endstconditionlbs})
and $(b)$ uses $U_{\tau-1} = \tau - 1 \geq \frac{e^4}{k} \log 2 - 1$.

Substituting this and \eqref{contribminimalinenvelope} into \eqref{eq:mastereqlbredux} gives
$1 \leq \lrp{ \frac{1}{2 \delta} } P (\tau < T) + \frac{1}{2}$.
Therefore, $P (\tau < T) \geq \delta$, finishing the proof. 
\end{proof}

The proof of Theorem \ref{thm:lbinterestingregime} requires a supporting lemma, which is proved in Section \ref{sec:lbotherlems}.
\begin{lem}
\label{lem:nostlbdelta}
Within the event $\{\tau(T) = T\}$,
if $\delta \geq \frac{4}{ \log \lrp{ k U_{T-1} }}$,
then $$\frac{1}{\sqrt{\delta \log \lrp{ \frac{2k U_{T-1}}{\abs{M_{T-1}} + 2 \sqrt{k U_{T-1}} } } }} \leq 1$$
\end{lem}

\subsection{Supporting Proofs}
\label{sec:pfoflem:lbintestimate}

\begin{proof}[Proof of Lemma \ref{lem:lbintestimate}]
First note that
\begin{align}
\mathbb{E}^{\lambda} \left[ Z_t^\lambda \right] &= \int_{-1/e^2}^{0} \expp{ \lambda M_t - k \lambda^2 U_t } \frac{d \lambda}{- \lambda \lrp{ \log \frac{1}{- \lambda} }^2} + 
\int_{0}^{1/e^2} \expp{ \lambda M_t - k \lambda^2 U_t } \frac{d \lambda}{\lambda \lrp{ \log \frac{1}{\lambda} }^2} \nonumber \\
&= \expp{ \frac{M_t^2}{4k U_t} } \Bigg[ \int_{-1/e^2}^{0} e^{ - k U_t \lrp{\lambda - \frac{M_t}{2k U_t}}^2 } \frac{d \lambda}{- \lambda \lrp{ \log \frac{1}{- \lambda} }^2} + 
\int_{0}^{1/e^2} e^{ - k U_t \lrp{\lambda - \frac{M_t}{2k U_t}}^2 } \frac{d \lambda}{\lambda \lrp{ \log \frac{1}{\lambda} }^2} \Bigg] \nonumber \\
\label{baseintlb}
&\leq 2 \expp{ \frac{M_t^2}{4k U_t} } \int_{0}^{1/e^2} e^{ - k U_t \lrp{\lambda - \frac{\abs{M_t}}{2k U_t}}^2 } \frac{d \lambda}{\lambda \lrp{ \log \frac{1}{\lambda} }^2}
\end{align}
Suppose that $\frac{\abs{M_t}}{2k U_t} \leq \frac{1}{e^2}$. 
Define an integer $N$ such that 
\begin{align}
\label{defofmeshnum}
\frac{\abs{M_t}}{2k U_t} + \sqrt{\frac{N-1}{k U_t}} \leq \frac{1}{e^2} \leq \frac{\abs{M_t}}{2k U_t} + \sqrt{\frac{N}{k U_t}}
\end{align}
(Note that we are guaranteed $N \geq 1$ by the assumption $\frac{\abs{M_t}}{2k U_t} \leq \frac{1}{e^2}$.) 
Combining this with \eqref{baseintlb},
\begin{align}
\label{intexpansion}
\mathbb{E}^{\lambda} \left[ Z_t^\lambda \right] \leq 2 \expp{ \frac{M_t^2}{4k U_t} } 
\Bigg[ &\int_{0}^{\frac{\abs{M_t}}{2k U_t}} e^{ - k U_t \lrp{\lambda - \frac{\abs{M_t}}{2k U_t}}^2 } \frac{d \lambda}{\lambda \lrp{ \log \frac{1}{\lambda} }^2} \nonumber \\
&+ 
\sum_{i=0}^{N-1} \int_{\frac{\abs{M_t}}{2k U_t} + \sqrt{\frac{i}{k U_t}} }^{\frac{\abs{M_t}}{2k U_t} + \sqrt{\frac{i+1}{k U_t}} } 
e^{ - k U_t \lrp{\lambda - \frac{\abs{M_t}}{2k U_t}}^2 } \frac{d \lambda}{\lambda \lrp{ \log \frac{1}{\lambda} }^2} \Bigg] 
\end{align}
Now we have
\begin{align*}
\int_{0}^{\frac{\abs{M_t}}{2k U_t}} e^{ - k U_t \lrp{\lambda - \frac{\abs{M_t}}{2k U_t}}^2 } \frac{d \lambda}{\lambda \lrp{ \log \frac{1}{\lambda} }^2}
&\leq \int_{0}^{\frac{\abs{M_t}}{2k U_t}} \frac{d \lambda}{\lambda \lrp{ \log \frac{1}{\lambda} }^2}
= \left[ \frac{1}{ \log \frac{1}{\lambda} } \right]_{0}^{\frac{\abs{M_t}}{2k U_t}} 
= \frac{1}{ \log \frac{2k U_t}{\abs{M_t}} }
\end{align*}
Substituting this and Lemma \ref{lemlbpartaest} into \eqref{intexpansion}, we get 
\begin{align*}
\mathbb{E}^{\lambda} \left[ Z_t^\lambda \right]
&= 2 \expp{ \frac{M_t^2}{4k U_t} } \left[ \frac{1}{ \log \frac{2k U_t}{\abs{M_t}} } + 
\frac{1}{ \log \lrp{ \frac{2k U_t}{\abs{M_t} + 2 \sqrt{k U_t} } } } \lrp{ \frac{e}{e-1} + \frac{3}{2 \log \lrp{\frac{e^2}{ 1 + e^2/\sqrt{k} }} }  } \right]  \\
&\leq \frac{ \expp{ \frac{M_t^2}{4k U_t} }}{ \log \lrp{ \frac{2k U_t}{\abs{M_t} + 2 \sqrt{k U_t} } } } \lrp{ 6 + \frac{3}{ \log \lrp{\frac{e^2}{ 1 + e^2/\sqrt{k} }} }  } 
\leq \frac{ 15 \expp{ \frac{M_t^2}{4k U_t} }}{ \log \lrp{ \frac{2k U_t}{\abs{M_t} + 2 \sqrt{k U_t} } } } = G_t
\end{align*}
which proves the result when $\abs{M_t} \leq \frac{2k U_t}{e^2}$. 

Alternatively, if $\abs{M_t} > \frac{2k U_t}{e^2}$, from \eqref{baseintlb} we have
\begin{align*}
\mathbb{E}^{\lambda} \left[ Z_t^\lambda \right]
&\leq 2 \expp{ \frac{M_t^2}{4k U_t} } 
\int_{0}^{1/e^2} e^{ - k U_t \lrp{\lambda - \frac{\abs{M_t}}{2k U_t}}^2 } \frac{d \lambda}{\lambda \lrp{ \log \frac{1}{\lambda} }^2} \\
&\stackrel{(a)}{\leq} 2 e^2 \expp{ \frac{M_t^2}{4k U_t} } 
\lrp{ \int_{0}^{1/e^2} e^{ - k U_t \lrp{\lambda - \frac{\abs{M_t}}{2k U_t}}^2 } d \lambda }
\lrp{ \int_{0}^{1/e^2} \frac{d \lambda}{\lambda \lrp{ \log \frac{1}{\lambda} }^2} } \\
&= e^2 \expp{ \frac{M_t^2}{4k U_t} } 
\lrp{ \int_{0}^{1/e^2} e^{ - k U_t \lrp{\lambda - \frac{\abs{M_t}}{2k U_t}}^2 } d \lambda } \\
&\leq e^2 \expp{ \frac{M_t^2}{4k U_t} } 
\lrp{ \int_{-\infty}^{\frac{\abs{M_t}}{2k U_t}} e^{ - k U_t \lrp{\lambda - \frac{\abs{M_t}}{2k U_t}}^2 } d \lambda } \\
&= e^2 \expp{ \frac{M_t^2}{4k U_t} } \lrp{ \frac{\sqrt{\pi}}{2 \sqrt{k U_t}} }
\stackrel{(b)}{\leq} e^2 \expp{ \frac{M_t^2}{4k U_t} } \frac{ \sqrt{\pi}}{2 \log (\sqrt{k U_t})} 
\leq \frac{7 \expp{ \frac{M_t^2}{4k U_t} } }{ \log \lrp{\frac{2k U_t}{2 \sqrt{k U_t}}} } %\\
%&\leq \frac{e^2}{2} \alpha \sqrt{\pi} \frac{\expp{ \frac{M_t^2}{4k U_t} } }{\log \lrp{\frac{2k U_t}{\abs{M_t} + 2 \sqrt{k U_t}}} } \leq G_t
\end{align*}
where $(a)$ is by Chebyshev's integral inequality (Lemma \ref{lem:chebyintineq}), 
and $(b)$ is because $\frac{\log (\sqrt{k U_t})}{ \sqrt{k U_t}} \leq 1$.
\end{proof}

\begin{lem}
\label{lemlbpartaest}
Define $N$ as in \eqref{defofmeshnum}. 
Then %if $\frac{\abs{M_t}}{U_t} \leq \frac{1}{e^2 \lrp{1 + \sqrt{1/3}}}$,
$$\sum_{i=0}^{N-1} \int_{\frac{\abs{M_t}}{2k U_t} + \sqrt{\frac{i}{k U_t}} }^{\frac{\abs{M_t}}{2k U_t} + \sqrt{\frac{i+1}{k U_t}} } 
e^{ - k U_t \lrp{\lambda - \frac{\abs{M_t}}{2k U_t}}^2 } \frac{d \lambda}{\lambda \lrp{ \log \frac{1}{\lambda} }^2} 
\leq \frac{1}{ \log \lrp{ \frac{2k U_t}{\abs{M_t} + 2 \sqrt{k U_t} } } } \lrp{ \frac{e}{e-1} + \frac{3}{2 \log \lrp{\frac{e^2}{ 1 + e^2/\sqrt{k} }} }  }$$
\end{lem}

\begin{proof}
For convenience, define $\mu := \frac{\abs{M_t}}{2k U_t}$ and $\sigma := \frac{1}{\sqrt{k U_t}}$. 
In particular, this means that $\mu + \sigma \sqrt{N-1} \leq \frac{1}{e^2} \leq \mu + \sigma \sqrt{N}$.

\begin{align}
\sum_{i=0}^{N-1} &\int_{\frac{\abs{M_t}}{2k U_t} + \sqrt{\frac{i}{k U_t}} }^{\frac{\abs{M_t}}{2k U_t} + \sqrt{\frac{i+1}{k U_t}} } 
e^{ - k U_t \lrp{\lambda - \frac{\abs{M_t}}{2k U_t}}^2 } \frac{d \lambda}{\lambda \lrp{ \log \frac{1}{\lambda} }^2} 
= \sum_{i=0}^{N-1} \int_{\mu + \sigma \sqrt{i}}^{\mu + \sigma \sqrt{i+1}} 
e^{ - \frac{\lrp{\lambda - \mu}^2 }{\sigma^2} } \frac{d \lambda}{\lambda \lrp{ \log \frac{1}{\lambda} }^2} \nonumber \\
&\leq \sum_{i=0}^{N-1} e^{ - \frac{\lrp{\sigma \sqrt{i}}^2 }{\sigma^2} } 
\int_{\mu + \sigma \sqrt{i}}^{\mu + \sigma \sqrt{i+1}} \frac{d \lambda}{\lambda \lrp{ \log \frac{1}{\lambda} }^2} 
= \sum_{i=0}^{N-1} e^{ - i } 
\lrp{ \frac{1}{\log \lrp{\frac{1}{ \mu + \sigma \sqrt{i+1} } }} - \frac{1}{\log \lrp{\frac{1}{ \mu + \sigma \sqrt{i} } }} } \nonumber \\
&\leq \sum_{i=0}^{N-1} \frac{e^{ - i }}{\log \lrp{\frac{1}{ \mu + \sigma \sqrt{i+1} }} } 
= \sum_{i=0}^{N-1} e^{ - i } \left( \frac{1}{ \log \lrp{\frac{1}{ \mu + \sigma } } } 
+ \frac{\log \lrp{\frac{\mu + \sigma \sqrt{i+1} }{ \mu + \sigma }} }{ \log \lrp{\frac{1}{ \mu + \sigma } }  \log \lrp{\frac{1}{ \mu + \sigma \sqrt{i+1} }} } \right) \nonumber \\
&\leq \frac{1}{ \log \lrp{\frac{1}{ \mu + \sigma } } } \sum_{i=0}^{N-1} e^{ - i } 
\left( 1 + \frac{\log \lrp{1 + \frac{\sigma (\sqrt{i+1} - 1) }{ \mu + \sigma }} }{ \log \lrp{\frac{1}{ \mu + \sigma \sqrt{N} }} } \right) \nonumber \\
&\leq \frac{1}{ \log \lrp{\frac{1}{ \mu + \sigma } } } \sum_{i=0}^{N-1} e^{ - i } 
\left( 1 + \frac{\log \lrp{1 + \frac{\sigma \sqrt{i} }{ \mu + \sigma }} }{ \log \lrp{\frac{1}{ \frac{1}{e^2} + \sigma }} } \right) \nonumber \\
&\leq \frac{1}{ \log \lrp{\frac{1}{ \mu + \sigma } } } \lrp{ \sum_{i=0}^{\infty} e^{ - i } + \sum_{i=0}^{\infty} e^{ - i } 
\left( \frac{\log \lrp{1 + \sqrt{i} } }{ \log \lrp{\frac{e^2}{ 1 + e^2 \sigma }} } \right) } \nonumber \\
\label{denomsigmabound}
&\leq \frac{1}{ \log \lrp{\frac{1}{ \mu + \sigma } } } 
\lrp{ \frac{e}{e-1} + \frac{1}{\log \lrp{\frac{e^2}{ 1 + e^2/\sqrt{k} }} } \sum_{i=0}^{\infty} e^{ - i } \log \lrp{1 + \sqrt{i} } } \nonumber \\
&\leq \frac{1}{ \log \lrp{\frac{1}{ \mu + \sigma } } } \lrp{ \frac{e}{e-1} + \frac{3}{2 \log \lrp{\frac{e^2}{ 1 + e^2/\sqrt{k} }} }  }
\end{align}
where \eqref{denomsigmabound} follows from Lemma \ref{lem:infsumub}.
\end{proof}

\begin{lem}
\label{lem:infsumub}
$\sum_{i=0}^{\infty} e^{ - i } \log \lrp{1 + \sqrt{i} } \leq \frac{3}{2}$
\end{lem}
\begin{proof}
Take $f(x) = e^{ - x/2 } \log \lrp{1 + \sqrt{x} }$ for $x \geq 0$. 
Note $f'(x) = \frac{1}{2} e^{ - x/2 } \lrp{ \frac{1}{ \sqrt{x} (1 + \sqrt{x}) } - \log \lrp{1 + \sqrt{x} }}$. 
Since $\frac{1}{ \sqrt{x} (1 + \sqrt{x}) }$ is monotone decreasing and $\log \lrp{1 + \sqrt{x} }$ is monotone increasing, 
$f'(x)$ has exactly one root, corresponding to the maximum of $f(x)$. 
This can be numerically confirmed to occur at $x^* \approx 0.745$, 
and $f(x^*) \leq 0.5$. 

So $\sum_{i=0}^{\infty} e^{ - i } \log \lrp{1 + \sqrt{i} } \leq \sum_{i=0}^{\infty} e^{ - i/2 } \lrp{ e^{ - i/2 } \log \lrp{1 + \sqrt{i} } } 
\leq \frac{1}{2} \sum_{i=0}^{\infty} e^{ - i/2 } = \frac{\sqrt{e}}{2 (\sqrt{e} - 1)} \leq \frac{3}{2}$.
\end{proof}

\subsection{Ancillary Results and Proofs}
\label{sec:lbotherlems}

\begin{proof}[Proof of Lemma \ref{lem:cannotincreasemuch}]
Note that by definition of $\tau$,
\begin{align}
\label{endstconditionagain}
\abs{M_{\tau-1}} \leq \frac{2k }{e^2 } U_{\tau-1}
\qquad \mbox{and} \qquad
\abs{M_{\tau-1}} \leq \sqrt{ 2k U_{\tau-1} \log \left( \frac{ \log \lrp{ \frac{2k U_{\tau-1} }{\abs{M_{\tau-1}} + 2 \sqrt{k U_{\tau-1}} } } }{C_1 \delta} \right)}
\end{align}

Firstly, we have
\begin{align}
\label{ubconsecexp}
\frac{\expp{ \frac{M_{\tau}^2}{4k U_{\tau}} } }{\expp{ \frac{M_{\tau - 1}^2}{4k U_{\tau - 1}} }} 
\leq \frac{\expp{ \frac{M_{\tau}^2}{4k U_{\tau - 1}} } }{\expp{ \frac{M_{\tau - 1}^2}{4k U_{\tau - 1}} }} 
= \expp{ \frac{2 M_{\tau - 1} \xi_{\tau} + \xi_{\tau}^2 }{4k U_{\tau - 1}} } 
\stackrel{(a)}{\leq} \expp{ \frac{1}{e^2} + \frac{1}{144} } \leq \frac{7}{6}
\end{align}
where $(a)$ is because $M_{\tau - 1} \xi_{\tau} \leq \abs{M_{\tau - 1}} \leq \frac{2k }{e^2 } U_{\tau-1}$ by \eqref{endstconditionagain}, 
and because $\frac{\xi_{\tau}^2}{4k U_{\tau - 1}} = \frac{1}{4k U_{\tau - 1}} \leq \frac{1}{4k (108)} = \frac{1}{144}$.

We write $\tau(T)$ more concisely as $\tau$ here, 
and note that the minimum value of $T$ implies that $\tau(T) \geq 108$, 
a fact we will use throughout the proof.
Also, we have
\begin{align}
\label{logratub}
\frac{ \log \lrp{ \frac{2k U_{\tau - 1}}{\abs{M_{\tau - 1}} + 2 \sqrt{k U_{\tau - 1}} } } }{ \log \lrp{ \frac{2k U_{\tau}}{\abs{M_{\tau}} + 2 \sqrt{k U_{\tau}} } } }
&\stackrel{(a)}{\leq} \frac{ \log \lrp{ \frac{2k U_{\tau}}{\abs{M_{\tau - 1}} + 2 \sqrt{k U_{\tau}} } } }{ \log \lrp{ \frac{2k U_{\tau}}{\abs{M_{\tau}} + 2 \sqrt{k U_{\tau}} } } } 
\leq \frac{ \log \lrp{ \frac{2k U_{\tau}}{\abs{M_{\tau - 1}} + 2 \sqrt{k U_{\tau}} } } }{ \log \lrp{ \frac{2k U_{\tau}}{ 1 + \abs{M_{\tau - 1}} + 2 \sqrt{k U_{\tau}} } } } \nonumber \\
&\leq \frac{ \log \lrp{ \frac{2k U_{\tau}}{\abs{M_{\tau - 1}} + 2 \sqrt{k U_{\tau}} } } }{ \log \lrp{ \frac{2k U_{\tau}}{\abs{M_{\tau - 1}} + 2 \sqrt{k U_{\tau}} }} - \log (13/12) }
\end{align}
where $(a)$ is because $f(x) = \log \lrp{\frac{x}{C + \sqrt{2x}}}$ is monotone increasing for any $C \geq 0$; 
and $(b)$ is because $\abs{M_{\tau - 1}} + 2 \sqrt{k U_{\tau}} \geq 2 \sqrt{k U_{\tau}} \geq 12$, 
so $1 + \abs{M_{\tau - 1}} + 2 \sqrt{k U_{\tau}} \leq \frac{13}{12} (\abs{M_{\tau - 1}} + 2 \sqrt{k U_{\tau}})$.

Now we have 
\begin{align*}
\frac{2k U_{\tau}}{\abs{M_{\tau - 1}} + 2 \sqrt{k U_{\tau}} } 
&\stackrel{(b)}{\geq} \frac{2k U_{\tau}}{  \frac{2k }{e^2 } U_{\tau-1} + 2 \sqrt{k U_{\tau}} } 
\geq \frac{k U_{\tau}}{ \max \left( \frac{2k }{e^2 } U_{\tau-1}, 2 \sqrt{k U_{\tau}} \right) } \\
&\geq \min \left( \frac{e^2 }{2k } , \frac{1}{2} \sqrt{k U_{\tau}} \right) \stackrel{(c)}{\geq} 3
\end{align*}
where $(b)$ uses \eqref{endstconditionagain} and $(c)$ uses $U_{\tau} = \tau \geq 108$.
This means that from \eqref{logratub},
\begin{align}
\label{ratiologub}
\frac{ \log \lrp{ \frac{2k U_{\tau - 1}}{\abs{M_{\tau - 1}} + 2 \sqrt{k U_{\tau - 1}} } } }{ \log \lrp{ \frac{2k U_{\tau}}{\abs{M_{\tau}} + 2 \sqrt{k U_{\tau}} } } }
\leq \frac{ \log (3) }{ \log (3) - \log (13/12) } \leq \frac{12}{11}
\end{align}

Note that by \eqref{endstconditionagain}, $\abs{M_{\tau - 1}} \leq \frac{2k}{e^2} U_{\tau - 1}$, 
so $G_{\tau - 1} = \frac{15 \expp{ \frac{M_{\tau - 1}^2}{4k U_{\tau - 1}} }}{ \log \lrp{ \frac{2k U_{\tau - 1}}{\abs{M_{\tau - 1}} + 2 \sqrt{k U_{\tau - 1}} } } } $.

Suppose $\abs{M_{\tau}} \leq \frac{2k}{e^2} U_{\tau}$. 
Then using Lemma \ref{lem:lbintestimate}, \eqref{ubconsecexp}, and \eqref{ratiologub},
\begin{align*}
\frac{G_{\tau}}{G_{\tau - 1}} &= \frac{\expp{ \frac{M_{\tau}^2}{4k U_{\tau}} } }{\expp{ \frac{M_{\tau - 1}^2}{4k U_{\tau - 1}} }} 
\frac{ \log \lrp{ \frac{2k U_{\tau - 1}}{\abs{M_{\tau - 1}} + 2 \sqrt{k U_{\tau - 1}} } } }{ \log \lrp{ \frac{2k U_{\tau}}{\abs{M_{\tau}} + 2 \sqrt{k U_{\tau}} } } }
\leq \lrp{\frac{7}{6}} \frac{12}{11} = \frac{14}{11}
\end{align*}
Alternatively, if $\abs{M_{\tau}} > \frac{2k}{e^2} U_{\tau}$, we can use Lemma \ref{lem:lbintestimate} and \eqref{ubconsecexp} to conclude that
\begin{align*}
\frac{G_{\tau}}{G_{\tau - 1}} &= \frac{7 \expp{ \frac{M_{\tau}^2}{4k U_{\tau}} } }{15 \expp{ \frac{M_{\tau - 1}^2}{4k U_{\tau - 1}} }} 
\frac{ \log \lrp{ \frac{2k U_{\tau - 1}}{\abs{M_{\tau - 1}} + 2 \sqrt{k U_{\tau - 1}} } } }{ \log \lrp{ \sqrt{k U_{\tau}} } } \\
&\leq \frac{\expp{ \frac{M_{\tau}^2}{4k U_{\tau}} } }{\expp{ \frac{M_{\tau - 1}^2}{4k U_{\tau - 1}} }} 
\frac{ \log \lrp{ \sqrt{k U_{\tau - 1}} } }{ \log \lrp{ \sqrt{k U_{\tau}} } }
\leq \frac{7}{6} (1) \leq \frac{14}{11}
\end{align*}
\end{proof}

%\begin{thm}
%Fix some time $t$. 
%Then w.p. $\geq \delta$, 
%$ \abs{M_t} > \sqrt{ 2 k^2 U_{t} \log \left( \frac{ 2 }{\delta} \right)} $.
%\end{thm}
%\begin{proof}[Proof of Theorem \ref{thm:lbboringregime}]
%From Lemma \ref{lem:submartconstr},
%\begin{align*}
%1 = \evp{Z_0^\lambda}{} \geq \evp{Z_t^\lambda}{} \geq 
%\end{align*}
%\end{proof}

\begin{proof}[Proof of Lemma \ref{lem:nostlbdelta}]
The definition of $\tau (= T)$ and the fact that $\abs{M_{\tau-1}} \geq 0$ imply that
\begin{align*}
\abs{M_{T-1}} &= \abs{M_{\tau-1}} \leq \sqrt{ 2k U_{\tau-1} \log \left( \frac{ \log \lrp{ \frac{2k U_{\tau-1} }{\abs{M_{\tau-1}} + 2 \sqrt{k U_{\tau-1}} } } }{C_1 \delta} \right)} \\
&\leq \sqrt{ 2k U_{T-1} \log \left( \frac{ \log \lrp{ k U_{T-1} } }{2 C_1 \delta} \right)}
\end{align*}
Consequently,
\begin{align}
\label{logratiolbdelta}
\log &\lrp{ \frac{2k U_{T-1}}{\abs{M_{T-1}} + 2 \sqrt{k U_{T-1}} } }
\geq \log \lrp{ \frac{2k U_{T-1}}{ \sqrt{ 2k U_{T-1} \log \left( \frac{ \log \lrp{ k U_{T-1} } }{2 C_1 \delta} \right)} + 2 \sqrt{k U_{T-1}} } } \nonumber \\
&= \log \lrp{ \frac{\sqrt{2k U_{T-1}}}{ \sqrt{ \log \left( \frac{ \log \lrp{ k U_{T-1} } }{2 C_1 \delta} \right)} + \sqrt{2} } } 
\stackrel{(a)}{\geq} \log \lrp{ \frac{\sqrt{2k U_{T-1}}}{ \sqrt{ \log \left( \log^2 \lrp{ k U_{T-1} } \right)} + \sqrt{2} } } \nonumber \\
&\stackrel{(b)}{\geq} \log \lrp{ \frac{\sqrt{k U_{T-1}}}{ \lrp{ k U_{T-1} }^{1/16} + 1 } }
\stackrel{(c)}{\geq} \frac{1}{4} \log \lrp{ k U_{T-1} }
\end{align}
where $(a)$ uses that $\delta \geq \frac{4}{\log \lrp{ k U_{T-1} }}$ and $8 C_1 \geq 1$, 
$(b)$ uses that $\log \left( \log \lrp{ k U_{T-1} } \right) \leq (k U_{T-1})^{1/8} $, 
and $(c)$ uses that $(k U_{T-1})^{1/16} + 1 \leq (k U_{T-1})^{1/4}$ for $k U_{T-1} = k (\tau - 1) \geq e^4 \ln(2/\delta) - k \geq 26$. 
Therefore, 
\begin{align*}
\frac{1}{\sqrt{\delta \log \lrp{ \frac{2k U_{T-1}}{\abs{M_{T-1}} + 2 \sqrt{k U_{T-1}} } } }} 
\leq \frac{1}{\sqrt{\frac{ \delta}{4} \log \lrp{ k U_{T-1} } }}
\leq 1
\end{align*}
after substituting \eqref{logratiolbdelta} and again using $\delta \geq \frac{4}{ \log \lrp{ k U_{T-1} }}$.
\end{proof}

Chebyshev's Integral Inequality is a standard result, 
but we give a short proof for completeness.
\begin{lem}[Chebyshev's Integral Inequality]
\label{lem:chebyintineq}
If $f(x)$ and $g(x)$ are respectively monotonically increasing and decreasing functions over an interval $(a,b]$, 
and $\int_a^b f(x) dx$ and $\int_a^b g(x) dx$ are both defined and finite, then
$$ \int_a^b f(x) g(x) dx \leq \frac{1}{b-a} \lrp{\int_a^b f(x) dx} \lrp{\int_a^b g(x) dx} $$
\end{lem}
\begin{proof}
By the monotonicity properties of the functions, we know for any $x, y \in (a,b]$ that $(f(x) - f(y)) (g(x) - g(y)) \leq 0$. 
Therefore,
\begin{align*}
0 &\geq \int_a^b \int_a^b (f(x) - f(y)) (g(x) - g(y)) \;dx\; dy \\
&= 2 (b-a) \int_a^b f(x) g(x) dx - 2 \lrp{\int_a^b f(x) dx} \lrp{\int_a^b g(x) dx}
\end{align*}
which yields the result upon simplification.
\end{proof}

%--------------------------------------------------------------------------------------------------------------------------------------------------------------------------------------------------------------------------------------
%---------------------------------------------------------------------------------------------------------------------------------------------------------------------------------------------------------------------------------------
%--------------------------------------------------------------------------------------------------------------------------------------------------------------------------------------------------------------------------------------
%---------------------------------------------------------------------------------------------------------------------------------------------------------------------------------------------------------------------------------------

\section{Miscellaneous Results: Concentration Bounds and Extensions}
\label{sec:intgsupermartlem1}

\subsection{Proof of Lemma \ref{lem:intgsupermartlem1}}

As outlined in Section \ref{sec:prelimadelta}, 
we choose $\lambda$ stochastically from a probability space $(\Omega_\lambda, \cF_\lambda, P_\lambda)$ such that 
$\displaystyle P_\lambda (d \lambda) = \frac{d \lambda}{\abs{\lambda} \lrp{ \log \frac{1}{\abs{\lambda}} }^2} $ on $\lambda \in [-e^{-2}, e^{-2}] \setminus \{ 0 \}$.

Take an arbitrary time $t \geq \tau_0 $.
For outcomes within $A_\delta$, we have the following: 
\begin{align}
\mathbb{E}^{\lambda} \left[ X_t^\lambda \right] 
&= \int_{-1/e^2}^{0} \expp{ \lambda M_t - \frac{\lambda^2}{2} U_t } \frac{d \lambda}{- \lambda \lrp{ \log \frac{1}{- \lambda} }^2} + 
\int_{0}^{1/e^2} \expp{ \lambda M_t - \frac{\lambda^2}{2} U_t } \frac{d \lambda}{\lambda \lrp{ \log \frac{1}{\lambda} }^2} \nonumber \\
&= \expp{ \frac{M_t^2}{2 U_t} } \Bigg[ \int_{-1/e^2}^{0} e^{ - \frac{1}{2} U_t \lrp{\lambda - \frac{M_t}{U_t}}^2 } \frac{d \lambda}{- \lambda \lrp{ \log \frac{1}{- \lambda} }^2} + 
\int_{0}^{1/e^2} e^{ - \frac{1}{2} U_t \lrp{\lambda - \frac{M_t}{U_t}}^2 } \frac{d \lambda}{\lambda \lrp{ \log \frac{1}{\lambda} }^2} \Bigg] \nonumber \\
\label{eq:casesintmix}
&\geq \expp{ \frac{M_t^2}{2 U_t} } 
\int_{0}^{1/e^2} e^{ - \frac{1}{2} U_t \lrp{\lambda - \frac{\abs{M_t}}{U_t}}^2 } \frac{d \lambda}{\lambda \lrp{ \log \frac{1}{\lambda} }^2} \\
\label{eq:crapintlb}
&\geq \expp{ \frac{M_t^2}{2 U_t} } \int_{\frac{\abs{M_t}}{U_t}\lrp{1 - \sqrt{k}} }^{ \frac{\abs{M_t}}{U_t}\lrp{1 + \sqrt{k}} } 
e^{ - \frac{1}{2} U_t \lrp{\lambda - \frac{\abs{M_t}}{U_t}}^2 } \frac{d \lambda}{\lambda \lrp{ \log \frac{1}{\lambda} }^2} \\
&\geq \expp{ \frac{M_t^2}{2 U_t} } \expp{ - \frac{1}{2} k U_t \lrp{ \frac{M_t}{U_t } }^2 } 
\int_{\frac{\abs{M_t}}{U_t}\lrp{1 - \sqrt{k}} }^{ \frac{\abs{M_t}}{U_t}\lrp{1 + \sqrt{k}} } 
\frac{d \lambda}{\lambda \lrp{ \log \frac{1}{\lambda} }^2} \nonumber \\
&= \expp{ \frac{M_t^2}{2 U_t} (1-k) } 
\frac{\log \lrp{\frac{1 + \sqrt{k}}{1 - \sqrt{k}} } } { \log \lrp{\frac{U_t}{\abs{M_t} \lrp{1 + \sqrt{k}}} } \log \lrp{\frac{U_t}{\abs{M_t} \lrp{1 - \sqrt{k}}} } } \nonumber \\
\label{laplacelbeqn}
&\geq 2 \expp{ \frac{M_t^2}{2 U_t} (1-k) } \frac{1 } { \log \lrp{\frac{U_t}{\abs{M_t} \lrp{1 + \sqrt{k}}} } \log \lrp{\frac{U_t}{\abs{M_t} \lrp{1 - \sqrt{k}}} } } \nonumber \\
&\geq \frac{ 2 \expp{ \frac{M_t^2}{2 U_t} (1-k) } } { \log^2 \lrp{\frac{U_t}{ \lrp{1 - \sqrt{k}} \abs{ M_t} } } }
\end{align}

Take $\tau$ to be any stopping time as in the lemma statement. Then from \eqref{laplacelbeqn},
\begin{align*}
\evp{ \mathbb{E}^{\lambda} \left[ X_{\tau}^\lambda \right] }{A_\delta}
\geq \evp{ \frac{ 2 \expp{ \frac{M_{\tau}^2}{2 U_{\tau}} (1-k) } } { \log^2 \lrp{\frac{U_{\tau}}{ \lrp{1 - \sqrt{k}} \abs{ M_{\tau}} } }} }{A_\delta} 
\end{align*}
finishing the proof.

\subsection{Hoeffding and Bernstein Concentration Bounds}
\label{sec:mainpfgen}

In this section, we show that the results of Section \ref{sec:unifbernbounds} can be proved 
through simple extensions of the proof of Theorem \ref{thm:newunifmart}.

That proof is the subject of Section \ref{sec:pfmainthm}. 
It applies to the Rademacher random walk, 
but uses the i.i.d. Rademacher assumption only through an exponential supermartingale construction (Lemma \ref{lem:supermartconstr}). 
Theorem \ref{thm:newunifmart} can be generalized significantly beyond the Rademacher random walk 
by simply replacing the construction with other similar exponential constructions, 
leaving the remainder of the proof essentially intact 
as presented in Section \ref{sec:pfmainthm}. 

To be specific, the rest of that proof works unchanged if the construction has the following properties:

\begin{enumerate}
\item The construction should be of the same form as Lemma \ref{lem:supermartconstr}: 
$X_t^\lambda = \expp{\lambda M_t - \frac{\lambda^2}{2} U_t }$ for some nondecreasing process $U_t$. 
(The proof of Theorem \ref{thm:newunifmart} sets $U_t = t$.)
\item $X_t^\lambda$ should be a supermartingale for $\lambda \in \lrp{- \frac{1}{e^2}, \frac{1}{e^2}} \setminus \{ 0 \}$.
\end{enumerate}

(The constant $\frac{1}{e^2}$ in these conditions is determined by the choice of averaging distribution over $\lambda$ 
(i.e., $P_\lambda$) in the proof of Lemma \ref{lem:intgsupermartlem1}. 
See Section \ref{sec:constschainatoms} for examples of other averaging distributions.)

Now we give two standard exponential supermartingale constructions with these properties. 
We first give a construction leading directly to Theorem \ref{thm:unifbmart}, 
when it is used to replace Lemma \ref{lem:supermartconstr} in the proof of Theorem \ref{thm:newunifmart}. 
\begin{lem}
\label{lem:bmartconstr}
Suppose the difference sequence is uniformly bounded, i.e. $\abs{\xi_t} \leq e^2$ a.s. for all $t$. 
Then the process $\displaystyle X_t^\lambda := \expp{ \lambda M_t - \lambda^2 (e-2) V_t }$ is a supermartingale 
for any $\lambda \in \left[ - \frac{1}{e^2} , \frac{1}{e^2} \right]$.
\end{lem}

\begin{proof}[Proof of Lemma \ref{lem:bmartconstr}]
It can be checked that $e^x \leq 1 + x + (e-2) x^2$ for $x \leq 1$. 
Then for any $\lambda \in \left[ - \frac{1}{e^2} , \frac{1}{e^2} \right]$ and $t \geq 1$, 
\begin{align*}
\evp{\expp{\lambda \xi_t} \mid \cF_{t-1}}{} &\leq 1 + \lambda \evp{\xi_t \mid \cF_{t-1}}{} + \lambda^2 (e-2) \evp{\xi_t^2 \mid \cF_{t-1}}{} \\
&= 1 + \lambda^2 (e-2) \evp{\xi_t^2 \mid \cF_{t-1}}{} \leq \expp{ \lambda^2 (e-2) \evp{\xi_t^2 \mid \cF_{t-1}}{}}
\end{align*}
using the martingale property on $\evp{\xi_t \mid \cF_{t-1}}{}$. 

Therefore, $\evp{\expp{\lambda \xi_t - \lambda^2 (e-2) \evp{\xi_t^2 \mid \cF_{t-1}}{} } \mid \cF_{t-1}}{} \leq 1$, so $\evp{X_t^\lambda \mid \cF_{t-1}}{} \leq X_{t-1}^\lambda$.
\end{proof}

The second construction leads similarly to Theorem \ref{thm:unifah}, 
when combined with the theorem's assumption of a uniformly bounded difference sequence; 
it has better constants than Lemma \ref{lem:bmartconstr}.

\begin{lem}
\label{lem:qmartconstr}
The process $\displaystyle X_t^\lambda := \expp{ \lambda M_t - \frac{\lambda^2}{6} \left( 2 V_t + Q_t \right)}$ is a supermartingale for any $\lambda \in \RR$.
\end{lem}

\begin{proof}[Proof of Lemma \ref{lem:qmartconstr}]
Consider the following inequality: for all real $x$,
\begin{align}
\label{expuppboundquad}
\expp{x - \frac{1}{6} x^2} \leq 1 + x + \frac{1}{3} x^2
\end{align}
Suppose \eqref{expuppboundquad} holds. 
Then for any $\lambda \in \RR$ and $t \geq 1$, 
$\evp{\expp{\lambda \xi_t - \frac{\lambda^2}{6} \xi_t^2} \mid \cF_{t-1}}{} \leq 1 + \lambda \evp{\xi_t \mid \cF_{t-1}}{} + \frac{\lambda^2}{3} \evp{\xi_t^2 \mid \cF_{t-1}}{} = 1 + 
\frac{\lambda^2}{3} \evp{\xi_t^2 \mid \cF_{t-1}}{} \leq \expp{\frac{\lambda^2}{3} \evp{\xi_t^2 \mid \cF_{t-1}}{}}$
, by using the martingale property on $\evp{\xi_t \mid \cF_{t-1}}{}$. 
Therefore, we have that 
$$\evp{\expp{\lambda \xi_t - \frac{\lambda^2}{6} \xi_t^2 - \frac{\lambda^2}{3} \evp{\xi_t^2 \mid \cF_{t-1}}{} } \mid \cF_{t-1}}{} \leq 1$$
so $\evp{X_t^\lambda \mid \cF_{t-1}}{} \leq X_{t-1}^\lambda$ and the result is shown.

It only remains to prove \eqref{expuppboundquad}, which is equivalent to showing that the function $f(x) = \expp{x - \frac{1}{6} x^2} - 1 - x - \frac{1}{3} x^2 \leq 0$. 
This is done by examining derivatives.
Note that $f'(x) = \lrp{1 - \frac{x}{3}} \expp{x - \frac{1}{6} x^2} - 1 - \frac{2}{3} x$, and
$f''(x) = \lrp{-\frac{1}{3} + \left(1 - \frac{x}{3} \right)^2} \expp{x - \frac{1}{6} x^2} - \frac{2}{3} = \frac{2}{3} \lrp{e^y \lrp{1 - y} - 1}$
where $y := x - \frac{1}{6} x^2$. 
Here $e^y \leq \frac{1}{1-y}$ for $y < 1$, and $e^y (1-y) \leq 0$ for $y \geq 1$, so $f''(x) \leq 0$ for all $x$. 
Since $f'(0) = f(0) = 0$, the function $f$ attains a maximum of zero over its domain, proving \eqref{expuppboundquad} and the result.
\end{proof}

In order for $X_t^\lambda$ to satisfy the conditions at the beginning of this section, 
the differences $\xi_t$ need not be uniformly bounded, but rather can simply satisfy conditions on their higher moments. 
For completeness, here is an example of much weaker sufficient conditions; 
this is a direct adaptation of the method conventionally used to prove a general version of Bernstein's inequality (\cite{BLM13}, Thm. 2.10).

\begin{lem}
\label{lem:bernubrelaxconstr}
%Define the cumulative $k^{th}$-moment process of the martingale $M_t$ as 
%$V_t^{(k)} = \sum_{i=1}^t \evp{\xi_i^k \mid \cF_{i-1}}{}$.
Suppose that for all $k \geq 3$, there is a constant $c> 0$ such that for all $i$, 
$\displaystyle \evp{\xi_i^k \mid \cF_{i-1}}{} \leq \frac{1}{2} k! (e/ \sqrt{2})^{2(k-2)} \evp{\xi_i^2 \mid \cF_{i-1}}{}$. 
Then for any $\lambda \in \lrp{- \frac{1}{e^2}, \frac{1}{e^2}}$,
the process 
$$ X_t^\lambda := \expp{ \lambda M_t - \lambda^2 V_t } $$
 is a supermartingale.
\end{lem}

\begin{proof}[Proof of Lemma \ref{lem:bernubrelaxconstr}]
Taking the Taylor expansion of the exponential function, 
\begin{align*}
\evp{\expp{\lambda \xi_t} \mid \cF_{t-1}}{} &= 1 + \lambda \evp{\xi_t \mid \cF_{t-1}}{} + 
\evp{ \sum_{k = 2}^{\infty} \frac{\lambda^k}{k!} \xi_t^{k} \mid \cF_{t-1}}{} \\
&= 1 + \frac{\lambda^2}{2} \evp{\xi_t^2 \mid \cF_{t-1}}{} + \sum_{k = 3}^{\infty} \frac{\lambda^k}{k!} \evp{ \xi_t^{k} \mid \cF_{t-1}}{} \\
&\leq 1 + \frac{\lambda^2}{2} \evp{\xi_t^2 \mid \cF_{t-1}}{} + \frac{1}{2} \evp{\xi_t^2 \mid \cF_{t-1}}{} \sum_{k = 3}^{\infty} \lambda^k (e/ \sqrt{2})^{2(k-2)} \\
&\leq 1 + \frac{\lambda^2}{2} \evp{\xi_t^2 \mid \cF_{t-1}}{} \sum_{k = 2}^{\infty} (\abs{\lambda/2} e^2)^{k-2}
\end{align*}
using the martingale property on $\evp{\xi_t \mid \cF_{t-1}}{}$, monotone convergence, and the lemma assumption. 
For all $\abs{\lambda} < \frac{2}{e^2}$, the infinite geometric series is summable, giving
$\displaystyle \evp{\expp{\lambda \xi_t} \mid \cF_{t-1}}{} \leq 1 + \frac{\lambda^2}{2 (1 - \abs{\lambda/2} e^2)} \evp{\xi_t^2 \mid \cF_{t-1}}{}$.
So for all $\abs{\lambda} < \frac{1}{e^2}$, 
\begin{align*}
\evp{\expp{\lambda \xi_t} \mid \cF_{t-1}}{} 
&\leq 1 + \frac{\lambda^2}{2 (1 - (1/2))} \evp{\xi_t^2 \mid \cF_{t-1}}{} 
\leq \expp{ \lambda^2 \evp{\xi_t^2 \mid \cF_{t-1}}{} }
\end{align*}
Therefore, 
%\begin{align*}
$\displaystyle
\evp{X_t^\lambda \mid \cF_{t-1}}{} 
\leq X_{t-1}^\lambda \evp{\expp{\lambda \xi_t - \lambda^2 \evp{\xi_t^2 \mid \cF_{t-1}}{}  } \mid \cF_{t-1}}{}
\leq X_{t-1}^\lambda
$.%\end{align*}
\end{proof}

\subsection{Initial Time Conditions}
\label{sec:inittime}
As discussed in Section \ref{sec:extensions} (in Remark \ref{remark:inittime}), 
this section outlines how to remove the initial time condition for our martingale concentration bounds, 
using the proof technique of this manuscript.
To demonstrate, we extend Theorem \ref{thm:genunifmart} here to hold over all times, 
using the placeholder variance process $U_t$ and other notation as in Section \ref{sec:pfmainthm} (e.g. $\lambda_0$). 
The Hoeffding- and Bernstein-style results of Section \ref{sec:unifbernbounds} can be extended in exactly the same way, by redefining $U_t$. 
%We can therefore use Lemma \ref{lem:btstrpmgf}.

%By examining Theorem \ref{thm:genunifmart} here, we see that  
It suffices to show a uniform concentration bound for $\abs{M_t}$ over 
$t < \tau_2 (\delta) := \min \left\{ s: U_s \geq \frac{2}{\lambda_0^2} \log \left( \frac{2}{\delta} \right) \right\}$. 

\begin{thm}
\label{thm:freedmanineq}
Fix any $\delta > 0$. With probability $\geq 1-\delta$, 
for all $t < \tau_2 (\delta) $ simultaneously, 
$$ \abs{M_{t}} \leq \frac{2}{\lambda_0} \log \lrp{\frac{2}{\delta}} $$
\end{thm}
\begin{proof}
Write $\tau_2 (\delta)$ as $\tau_2$ for convenience.
Define the stopping time 
$ \tau = \min \left\{ t \leq \tau_2 : \abs{M_t} > \frac{2}{\lambda_0} \log \lrp{\frac{2}{\delta}} \right\} $.
Then it suffices to prove that $P(\tau < \tau_2) \leq \delta$.

On the event $\{ \tau < \tau_2 \}$, we have $\abs{M_{\tau}} > \frac{2}{\lambda_0} \log \lrp{\frac{2}{\delta}}$ by definition of $\tau$. 
Therefore, using Lemma \ref{lem:btstrpmgf},
\begin{align*}
2 &\geq \evp { \expp{ \lambda_0 \abs{M_\tau} - \frac{\lambda_0^2}{2} U_\tau} }{} 
\geq \evp { \expp{ \lambda_0 \abs{M_\tau} - \frac{\lambda_0^2}{2} U_\tau} \mid \tau < \tau_2}{} P(\tau < \tau_2) \\
&\stackrel{(a)}{\geq} \evp { \expp{ \lambda_0 \lrp{\frac{2}{\lambda_0} \log \lrp{\frac{2}{\delta}}} - \frac{\lambda_0^2}{2} \lrp{\frac{2}{\lambda_0^2} \log \lrp{\frac{2}{\delta}}} } 
\mid \tau < \tau_2}{} P(\tau < \tau_2) \\
&= \evp { \expp{ \log \lrp{\frac{2}{\delta}} } }{} P(\tau < \tau_2) 
= \frac{2}{\delta} P(\tau < \tau_2)
\end{align*}
where $(a)$ uses that $\abs{M_{\tau}} > \frac{2}{\lambda_0} \log \lrp{\frac{2}{\delta}}$ when $\tau < \tau_2$, 
and that $U_\tau \leq \frac{2}{\lambda_0^2} \log \lrp{\frac{2}{\delta}}$ since $\tau < \tau_2$. 
Therefore, $P(\tau < \tau_2) \leq \delta$, as desired. 
\end{proof}

Taking a union bound of Theorem \ref{thm:freedmanineq} with Theorem \ref{thm:genunifmart} 
gives that w.h.p., for all $t$, 
\begin{align*}
\abs{M_t} \leq \cO{ \sqrt{ U_t \lrp{ \log \log U_t + \log \frac{1}{\delta} }} + \log \lrp{\frac{1}{\delta}} }
\end{align*}
This matches the rate of Bennett/Bernstein inequalities which hold for a fixed time \cite{BLM13}, 
except for an extra $\sqrt{\log \log U_t}$ factor on the Gaussian-regime term that accounts for the uniformity over time.

\subsection{Sharper Constants}
\label{sec:constschainatoms}

The leading proportionality constant on the iterated-logarithm term in Theorem \ref{thm:genunifmart} is $\sqrt{6}$, 
above the LIL's asymptotic $\sqrt{2}$. 
The reasons for this relate to the proof of 
Lemma \ref{lem:intgsupermartlem1} (in Section \ref{sec:intgsupermartlem1}).
There are two ways to tighten this proof, which together yield the optimal $\sqrt{2}$ constant. 

First, the mixed process $\mathbb{E}^{\lambda} \left[ X_t^\lambda \right]$ in this proof 
can be written as a probability integral of a Gaussian-like function, 
which we crudely lower-bound around the peak (\eqref{eq:crapintlb}). 
A more refined lower bound here leads to a sharper final result (by a factor of $\sqrt{3/2}$). 

Second, the mixing distribution over $\lambda$ can be refined, 
so that the bound of Lemma \ref{lem:intgsupermartlem1} has a $\log(\cdot)$ factor in its denominator instead of a $\log^2 (\cdot)$. 
This decreases the proportionality constant (by a further factor of $\sqrt{2}$), to optimality.

Specifically, we now present a countably infinite family of averaging distributions for $\lambda$ 
that generalize the one used in the main proof, and lead to optimal constants if used analogously instead in the proof. 
To describe this set, 
define $\log_v (x) = \underbrace{\log \log \dots \log}_{\text{$v$ times}} (x) $ 
and $\exp_v (x) = \underbrace{\exp \exp \dots \exp}_{\text{$v$ times}} (x)$ for $v = 1, 2, \dots$. 
The following family of probability distributions is indexed by $v$:
$$ P_\lambda^v (d \lambda) = \frac{d \lambda}{\abs{\lambda} \log_v \lrp{ \frac{1}{\abs{\lambda}}} 
\left[ \prod_{i=1}^v \log_i \lrp{ \frac{1}{\abs{\lambda}}} \right] } \quad \text{for} \quad \lambda \in \left[ - \frac{1}{\exp_v (2)}, \frac{1}{\exp_v (2)} \right] \setminus \{ 0 \} $$
Note that $P_\lambda^1$ is used to mix over $\lambda$ in the main proof of this paper. 

Suppose the distribution $P_\lambda^v$ is used in the proof for some $v$. 
The first stage of the proof to prove a uniform LLN bound analogous to Theorem \ref{thm:varunifub}. 
The constants will be different and the initial time condition more restrictive to account for the smaller support of $P_\lambda^v$ relative to $P_\lambda^1$
(instead of an initial time $\tau_0$ as defined in Theorem \ref{thm:genunifmart}, it will use $\tau_0^v \geq \tau_0$, with $\tau_0 := \tau_0^1$), 
but otherwise this step follows Section \ref{sec:pfmainthm} closely.

%For any $v$, using $P_\lambda^v$ to mix over $\lambda$ in our proof technique 
%(in conjunction with an appropriate LLN, like Theorem \ref{thm:varunifub}) 
%gives the result the optimal iterated-logarithm rate. 
%$P_\lambda^1, P_\lambda^2, \dots$ require progressively more stringent LLNs 
%(because the support of $P_\lambda^v$ decreases with increasing $v$) 
%but lead to progressively tighter derived LIL bounds.

Working within the ``good" $(1-\delta)$-probability event of the resulting LLN (call it $A_\delta^v$, 
so that $A_\delta := A_\delta^1$ is the event used in the main proof), 
the proof then requires a moment bound analogous to Lemma \ref{lem:intgsupermartlem1}. 
This is where the averaging distribution $P_\lambda^v$ plays a direct role, 
replacing $P_\lambda^1$ in the proof of Lemma \ref{lem:intgsupermartlem1}. 

To illustrate, we sketch a refined version of Lemma \ref{lem:intgsupermartlem1}, 
incorporating the changes we have just described in this subsection, 
to achieve optimal constants. 
[Note that for any $v$, $P_\lambda^v$ has a tractable closed-form antiderivative: 
for $\lambda > 0$, it is $ \lrp{\log_v \lrp{ \frac{1}{\lambda}}}^{-1}$].

\begin{lem}[Substitute for Lemma \ref{lem:intgsupermartlem1}]
\label{lem:bettermgfbound}
Suppose the averaging distribution over $\lambda$ is $P_\lambda^v$ for some $v \geq 1$.
Then for any stopping time $\tau \geq \tau_0^v$,
$$ \evp{\mathbb{E}^{\lambda} \left[ X_{\tau}^\lambda \right]}{A_{\delta}^v} 
\geq \evp{ \frac{\expp{ \frac{M_{\tau}^2}{2 U_{\tau}} }}
{e^{2} \max \left[ \log_{v}^{+} \lrp{\sqrt{U_t} } \left[ \prod_{i=1}^{v} \log_i^{+} \lrp{ \sqrt{U_t} } \right] , \; \log_{v}^{+} \lrp{\frac{U_t}{\abs{M_t}} } \right]} }{A_\delta^v} $$
where $\log_{i}^{+} (x) = \log_{i} [\max \lrp{\exp_{i} (1), x}]$.
\end{lem}
\begin{proof}
For any time $t$, first suppose $\frac{\abs{M_t}}{U_t} - \frac{2}{\sqrt{U_t}} \geq 0$, i.e. $\abs{M_t} \geq 2 \sqrt{U_t}$. 
Then 
\begin{align}
\label{eq:robbcvxityfo}
\int_{0}^{1/e^2} e^{ - \frac{1}{2} U_t \lrp{\lambda - \frac{\abs{M_t}}{U_t}}^2 } &P_\lambda^v \;d \lambda
\geq \int_{\frac{\abs{M_t}}{U_t} - \frac{1}{\sqrt{U_t}}}^{\frac{\abs{M_t}}{U_t}} 
e^{ - \frac{1}{2} U_t \lrp{\lambda - \frac{\abs{M_t}}{U_t}}^2 } P_\lambda^v \;d \lambda 
\geq e^{ - 1/2 } \int_{\frac{\abs{M_t}}{U_t} - \frac{1}{\sqrt{U_t}}}^{\frac{\abs{M_t}}{U_t}} P_\lambda^v \;d \lambda \nonumber \\
&= e^{ - 1/2 } \lrp{ \frac{1}{ \log_{v} \lrp{\frac{U_t}{\abs{M_t}} }} - \frac{1}{\log_{v} \lrp{\frac{U_t}{\abs{M_t} - \sqrt{U_t}} } } } \nonumber \\
&= e^{ - 1/2 } \left[ F_v \lrp{ \log (S_t) } - F_v \lrp{ \log (S_t) + \log \lrp{\alpha} } \right] 
%\geq e^{ - 1/2 } \frac{\log \lrp{2} } { \log^2 \lrp{\sqrt{U_t}} } \geq 4 e^{ - 1/2 }\frac{\log \lrp{2} } { \log^2 \lrp{U_t} }
\end{align}

where $S_t = \frac{U_t}{\abs{M_t}}$, 
$\displaystyle F_v (x) = \frac{1}{\log_{v-1} (x)}$, 
and $\alpha = \frac{\abs{M_t}}{\abs{M_t} - \sqrt{U_t}} \in \left[2 , \frac{\abs{M_t}}{\sqrt{U_t}} \right]$. 
Note that the derivative of $F$ is expressible as 
$\displaystyle F_v' (x) = - \frac{1}{x \log_{v-1} \lrp{ x} 
\left[ \prod_{i=1}^{v-1} \log_i \lrp{ x} \right] }$. 
$F_v (\cdot)$ is monotone decreasing and convex, 
so \eqref{eq:robbcvxityfo} can be lower-bounded to first order as follows.
\begin{align}
\label{supermmgftradeoff}
\mbox{Eq. } \eqref{eq:robbcvxityfo} 
&\geq e^{ - 1/2 } \log \lrp{\alpha} \lrp{ - F_v' \lrp{ \log (S_t) + \log \lrp{\alpha} } } \nonumber \\
&= e^{ - 1/2 } \log \lrp{\alpha} \lrp{\log \lrp{\alpha S_t} \log_{v-1} \lrp{ \log \lrp{\alpha S_t} } 
\left[ \prod_{i=1}^{v-1} \log_i \lrp{ \log \lrp{\alpha S_t} } \right] }^{-1} \nonumber \\
&\geq e^{ - 1/2 } \log(2) \lrp{ \log_{v} \lrp{\alpha S_t } \left[ \prod_{i=1}^{v} \log_i \lrp{ \alpha S_t } \right] }^{-1} \nonumber \\
&\geq e^{ - 1/2 } \log(2) \lrp{ \log_{v} \lrp{\sqrt{U_t} } \left[ \prod_{i=1}^{v} \log_i \lrp{ \sqrt{U_t} } \right] }^{-1} 
\end{align}

Alternatively, if $\frac{\abs{M_t}}{U_t} - \frac{2}{\sqrt{U_t}} \leq 0$, 
\begin{align*}
\int_{0}^{1/e^2} e^{ - \frac{1}{2} U_t \lrp{\lambda - \frac{\abs{M_t}}{U_t}}^2 } P_\lambda^v \;d \lambda
&\geq \int_{0}^{\frac{\abs{M_t}}{U_t}} e^{ - \frac{1}{2} U_t \lrp{\lambda - \frac{\abs{M_t}}{U_t}}^2 } P_\lambda^v \;d \lambda \\
&\geq e^{ - \frac{1}{2} U_t \lrp{\frac{2}{\sqrt{U_t}}}^2 } \int_{0}^{\frac{\abs{M_t}}{U_t}} P_\lambda^v \;d \lambda 
= \frac{e^{ -2 }} { \log_{v} \lrp{\frac{U_t}{\abs{M_t}} } }
\end{align*}
Putting the two cases together, 
\begin{align}
\label{eq:betterintbound}
\int_{0}^{1/e^2} &e^{ - \frac{1}{2} U_t \lrp{\lambda - \frac{\abs{M_t}}{U_t}}^2 } \frac{d \lambda}{\lambda \lrp{ \log \frac{1}{\lambda} }^2} \nonumber \\
&\geq e^{ - 2 } \min \left[ \lrp{ \log_{v}^{+} \lrp{\sqrt{U_t} } \left[ \prod_{i=1}^{v} \log_i^{+} \lrp{ \sqrt{U_t} } \right] }^{-1} , \; \lrp{\log_{v}^{+} \lrp{\frac{U_t}{\abs{M_t}} } }^{-1} \right]
\end{align}
Substituting \eqref{eq:betterintbound} into \eqref{eq:casesintmix} in the proof of Lemma \ref{lem:intgsupermartlem1} gives the result. 
\end{proof}

We can carry out the rest of the proof of Theorem \ref{thm:genunifmart}, using Lemma \ref{lem:bettermgfbound} instead of Lemma \ref{lem:intgsupermartlem1}, 
whereby it can be verified that the resulting uniform non-asymptotic LIL bound, 
for sufficiently high $t$, is at most
\begin{align}
\label{eq:hieroflogs}
\sqrt{2 U_t \lrp{\log \lrp{\frac{2 e^2}{\delta}} + 
\max \left[ \sum_{i=2}^{v+1} \log_i^{+} \lrp{ \sqrt{U_t} } + \log_{v+1}^{+} \lrp{\sqrt{U_t} } \;,\; \log_{v+1}^{+} \lrp{\frac{U_t}{\abs{M_t}} } \right]} } 
\end{align}
%which leads to an ``optimal" leading proportionality constant of $\sqrt{3}$ as $k \to \infty$, 
%as claimed. 

In particular, as $t \to \infty$ the $\log_2 (\cdot)$ term dominates, 
and it has an unimprovable leading constant of $\sqrt{2}$ for any $v \geq 2$.
Also, the $v=1$ case improves the result of Theorem \ref{thm:genunifmart} from a proportionality constant of $\sqrt{6}$ to one of $2$. 

A similar family of distributions to $\{P_\lambda^v\}_{v=1,2,\dots}$ was considered 
by Robbins and Siegmund (\cite{RS70}, Example 4) in a strictly asymptotic setting for Brownian motion, without the $\delta$ dependence; 
our parametrization is clearer for our purposes. 
Our arguments here generalize some of those made in that paper (\cite{RS70}, Section 4) to finite times. 
But we point out that the essential arguments linking weights in the ``scale" $\lambda$-domain to the 
iterated-logarithm order of growth of the moment were noticed as early as \cite{F43}, 
which links back to the remarkable ad hoc calculations of Erd\H{o}s \cite{E42} illuminating the exact rate of growth in \eqref{eq:hieroflogs}.

\section*{Acknowledgements}
I am very grateful to 
Sanjoy Dasgupta, Patrick Fitzsimmons, Yoav Freund, and Matus Telgarsky 
for instructive conversations.

%--------------------------------------------------------------------------------------------------------------------------------------------------------------------------------------------------------------------------------------
%---------------------------------------------------------------------------------------------------------------------------------------------------------------------------------------------------------------------------------------

%\newpage
\appendix

\bibliographystyle{imsart-number}
\bibliography{aop-unifmarts_mref}

% AOS,AOAS: If there are supplements please fill:
%\begin{supplement}[id=suppA]
%  \sname{Supplement A}
%  \stitle{Title}
%  \slink[doi]{10.1214/00-AOASXXXXSUPP}
%  \sdatatype{.pdf}" 
%  \sdescription{Some text}
%\end{supplement}

\end{document}